\documentclass[English]{amsart}%

\usepackage[utf8]{inputenc}
\usepackage{amsmath} %
\usepackage{amsthm}
\usepackage{amsfonts}
\usepackage{amssymb}
\usepackage{fullpage}
\usepackage{mathtools}
\usepackage{stmaryrd}
\usepackage{mathrsfs}
\usepackage[bookmarks,hidelinks,pdfencoding=unicode]{hyperref}

\usepackage[capitalize]{cleveref}
\usepackage{listings}
\usepackage{enumerate}
\usepackage{bbm}
\usepackage{accents}
\usepackage{quiver}
\usepackage{caption}
\usepackage{ifthen}
\usepackage{multirow}
\usepackage{graphicx}
\usepackage{adjustbox}
\usepackage{mdframed}
\usepackage{tikz}
\usetikzlibrary{decorations.pathreplacing}
\usetikzlibrary{decorations.pathmorphing}
\usetikzlibrary{math}
\usetikzlibrary{arrows}

\newtheorem{thm}{Theorem}[section]

\newtheorem{prop}[thm]{Proposition}
\newtheorem{lem}[thm]{Lemma}
\newtheorem{cor}[thm]{Corollary}
\newtheorem{fact}[thm]{Fact}

\newtheorem{quest}[thm]{Question}
\newtheorem{conj}[thm]{Conjecture}
\theoremstyle{definition}
\newtheorem{defn}[thm]{Definition}

\theoremstyle{remark}

\newcommand{\Rb}{\mathbb{R}}
\newcommand{\Nb}{\mathbb{N}}
\newcommand{\Zb}{\mathbb{Z}}
\newcommand{\Qb}{\mathbb{Q}}

\newcommand{\Rc}{\mathcal{R}}

\newcommand{\Cb}{\mathbb{C}}

\makeatletter \DeclareRobustCommand{\cset}{\@ifstar\star@cset\normal@cset}
\newcommand{\star@cset}[1]{{\left\llbracket#1\right\rrbracket}}
\newcommand{\normal@cset}[2][]{{\mathopen{#1\llbracket}#2\mathclose{#1\rrbracket}}}
\makeatother

\newcommand{\xbar}{\bar{x}}
\newcommand{\ybar}{\bar{y}}
\newcommand{\zbar}{\bar{z}}

\newcommand{\kbar}{\bar{k}}
\newcommand{\ellbar}{\bar{\ell}}
\newcommand{\mbar}{\bar{m}}
\newcommand{\nbar}{\bar{n}}

\newcommand{\IZF}{\ifmmode\mathsf{IZF}\else$\mathsf{IZF}$\fi}
\newcommand{\CZF}{\ifmmode\mathsf{CZF}\else$\mathsf{CZF}$\fi}
\newcommand{\ZF}{\ifmmode\mathsf{ZF}\else$\mathsf{ZF}$\fi}
\newcommand{\ZFC}{\ifmmode\mathsf{ZFC}\else$\mathsf{ZFC}$\fi}
\newcommand{\AC}{\ifmmode\mathsf{AC}\else$\mathsf{AC}$\fi}
\newcommand{\AD}{\ifmmode\mathsf{AC}\else$\mathsf{AD}$\fi}
\newcommand{\BZ}{\ifmmode\mathsf{BZ}\else$\mathsf{BZ}$\fi}
\newcommand{\GCH}{\ifmmode\mathsf{GCH}\else$\mathsf{GCH}$\fi}
\newcommand{\PA}{\ifmmode\mathsf{PA}\else$\mathsf{PA}$\fi}
\newcommand{\DC}{\ifmmode\mathsf{DC}\else$\mathsf{DC}$\fi}
\newcommand{\MP}{\ifmmode\mathsf{MP}\else$\mathsf{MP}$\fi}
\newcommand{\CT}{\ifmmode\mathsf{CT}\else$\mathsf{CT}$\fi}
\newcommand{\PAx}{\ifmmode\mathsf{PAx}\else$\mathsf{PAx}$\fi}
\newcommand{\VL}{\ifmmode{\mathsf{V}=\mathsf{L}}\else$\mathsf{V}=\mathsf{L}$\fi}

\DeclareMathOperator{\Ord}{Ord}

\newcommand{\Set}{\mathrm{Set}}

\newcommand{\LL}{L}
\newcommand{\VV}{V}

\newcommand{\olR}{\overline{\Rb}}

\newcommand{\bDelta}{\undertilde{\mathbf{\Delta}}} \newcommand{\bPi}{\undertilde{\mathbf{\Pi}}}
\newcommand{\bSigma}{\undertilde{\mathbf{\Sigma}}}
\newcommand{\lDelta}{\Delta}
\newcommand{\lPi}{\Pi}
\newcommand{\lSigma}{\Sigma}

\DeclareMathOperator*{\qqq}{qqq}

\newlength{\savedparindent}
\AtBeginDocument{\setlength{\savedparindent}{\parindent}}

\newcommand{\Lusin}{Lusin}
\def\ZFCModelFigure{
  \includegraphics{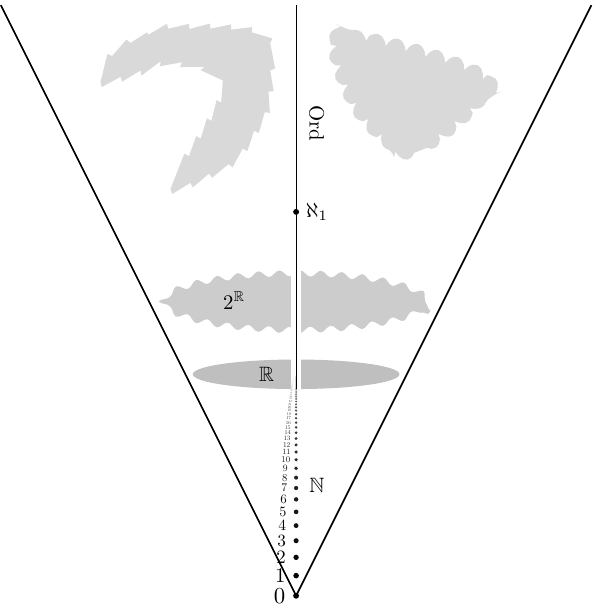}
}
\def\GodelLFigure{
  \includegraphics{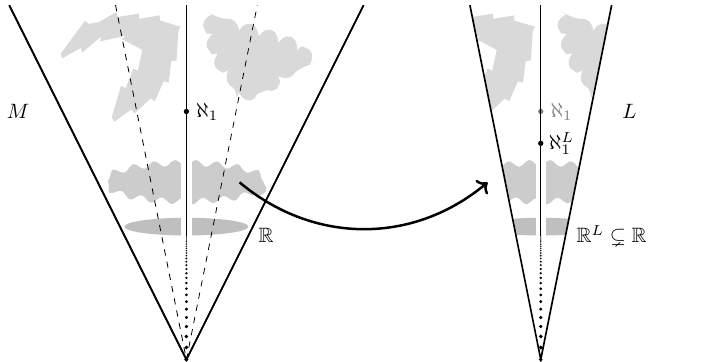}
}
\def\InaccFigure{
  \includegraphics{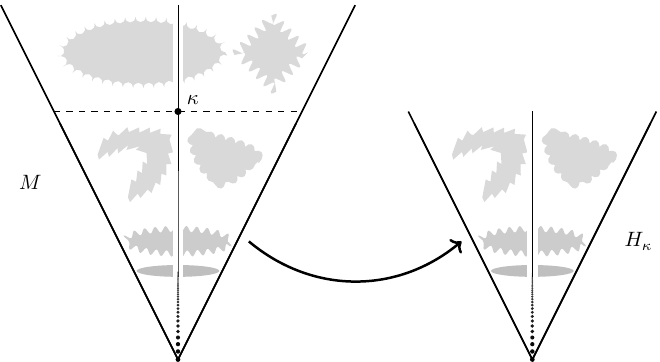}
}
\def\SolovayInaccFigure{
  \includegraphics{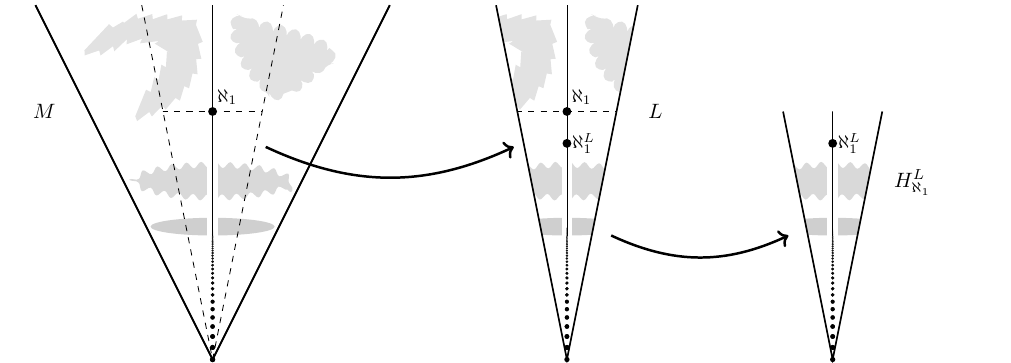}
}
\def\LRFigure{
  \includegraphics{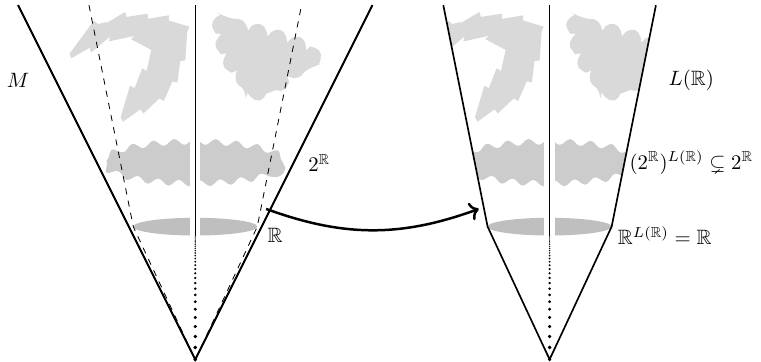}
}

\begin{document}

\title{Any function I can actually write down is measurable, right?} 

\address{Department of Mathematics \\
  Iowa State University \\
  396 Carver Hall \\
  411 Morrill Road \\
  Ames, IA 50011, USA}
\author{James E. Hanson}
\email{jameseh@iastate.edu}
\date{\today}

\keywords{independence results, inner model theory, effective descriptive set theory, non-measurable functions, universal Diophantine equations, large cardinals, o-minimality, paradoxical decompositions}
\subjclass[2020]{Primary 03-02, 03E35, 28A05, Secondary 03E45, 03E15, 03E55, 03C64}

\begin{abstract}
  In this expository paper aimed at a general mathematical audience, %
  we discuss how to combine certain classic theorems of set-theoretic inner model theory and effective descriptive set theory with work on Hilbert's tenth problem and universal Diophantine equations to produce the following surprising result: There is a specific polynomial $p(x,y,z,n,k_1,\dots,k_{70})$ of degree $7$ with integer coefficients such that it is independent of \ZFC\ (and much stronger theories) 
  whether the function
  \[
    f(x) = \inf_{y \in \Rb}\sup_{z \in \Rb}\inf_{n \in \Nb}\sup_{\kbar \in \Nb^{70}}p(x,y,z,n,\kbar)
  \]
  is Lebesgue measurable. %
  We also give similarly defined $g(x,y)$ with the property that the statement ``$x \mapsto g(x,r)$ is measurable for every $r \in \Rb$'' has large cardinal consistency strength (and in particular implies the consistency of \ZFC) and $h(m,x,y,z)$ such that $h(1,x,y,z),\allowbreak\dots,\allowbreak h(16,x,y,z)$ can consistently be the indicator functions of a Banach--Tarski paradoxical decomposition of the sphere. %

  Finally, we discuss some situations in which measurability of analogously defined functions can be concluded by inspection, which touches on %
 model-theoretic o-minimality and the fact that sufficiently strong large cardinal hypotheses (such as Vop\v{e}nka's principle and much weaker assumptions) imply that all `reasonably definable' functions (including the above $f(x)$, $g(x,y)$, and $h(m,x,y,z)$) are universally measurable.

\end{abstract}

\maketitle

\section*{Introduction}

While most students of mathematics encounter measure theory at some point, the basic machinery of measure theory (such as $\sigma$-algebras) is often regarded as fussy and irrelevant in practice. In particular, the issue of measurability of sets and functions seems to be something one \emph{only} encounters in the first semester of graduate real analysis and never again. The common informal explanation of this is that any function one can `actually write down' will be measurable and that one needs to go looking for trouble using the axiom of choice in order to find non-measurable objects. This is certainly a true statement empirically and one would like it to be true on a moral level, but there is nevertheless a sense in which it is technically false and we will be exploring the extent to which this statement is both in some sense false and in some sense true here.

On a purely mathematical level, the main contributions of this paper are Theorems~\ref{thm:polynomial-Vitali}, \ref{thm:polynomial-inacc}, and \ref{thm:Polynomial-Banach-Tarski}. Each of these is a statement of the same form: ``There is a degree $7$ polynomial $p$ with integer coefficients and fewer than $100$ variables such that writing some $\inf$'s and $\sup$'s over $\Nb$ and $\Rb$ in front of $p$ defines a function whose measurability is wildly independent of \ZFC.'' More specifically we have the following:
\begin{itemize}
\item (\cref{thm:polynomial-Vitali}) There is a degree $7$ polynomial $p(x,y,z,n,k_1,\dots,k_{70})$ with integer coefficients such that it is independent of \ZFC{} whether
  \[
    f(x) = \inf_{y \in \Rb}\sup_{z \in \Rb}\inf_{n \in \Nb}\sup_{\kbar \in \Nb^{70}}p(x,y,z,n,\kbar)
  \]
  is Lebesgue measurable.
\item (\cref{thm:polynomial-inacc}) There is a degree $7$ polynomial $p(x,y,z,w,t,n,k_1,\dots,k_{74})$ with integer coefficients such that the function
  \[
    g(x,y) = \sup_{z \in \Rb} \inf_{w \in \Rb}\sup_{t \in \Rb}\inf_{n \in \Nb}\sup_{\kbar \in \Nb^{74}}p(x,y,z,w,t,n,\kbar)
  \]
  has the property that the statement ``$g(x,r)$ is Lebesgue measurable for every $r \in \Rb$'' has the same consistency strength as the existence of an inaccessible cardinal (i.e., a Grothendieck universe).
\item (\cref{thm:Polynomial-Banach-Tarski}) There is a degree $7$ polynomial $p(m,x,y,z,w,t,n,k_1,\dots,k_{76})$ with integer coefficients such that if we define
  \[
    h(m,x,y,z) = \inf_{w \in \Rb} \sup_{t \in \Rb} \inf_{n \in \Nb} \sup_{\kbar \in \Nb^{76}}p(m,x,y,z,w,t,n,\kbar),
  \]
 then it is independent of \ZFC{} whether $h(1,x,y,z),\dots,h(16,x,y,z)$ are the indicator functions of a Banach--Tarski paradoxical decomposition of the unit sphere.
\end{itemize}
Although these results, to my\footnote{We will occasionally use the terms \emph{I}, \emph{me}, and \emph{my} as convenient shorthand for `the author' or `the author's,' as appropriate.} knowledge, have never been pointed out explicitly before, they are not morally new. \cref{thm:polynomial-inacc} could have been written down in 1984 as a corollary of Shelah's work showing that you `can't take Solovay's inaccessible away' \cite{Shelah1984}. If we forget about the specific numbers of variables given, Theorems~\ref{thm:polynomial-Vitali} and \ref{thm:Polynomial-Banach-Tarski} could have been written down in 1970, when Matiyasevich (building on earlier work of Robinson, Davis, and Putnam) solved Hilbert's tenth problem negatively \cite{Matiyasevich_1970}. If we moreover allow the polynomials to take some exponential functions as arguments (and if we only demand consistency with \ZFC{} rather than independence), Theorems~\ref{thm:polynomial-Vitali} and \ref{thm:Polynomial-Banach-Tarski} could have been written down in 1960, when Robinson, Davis, and Putnam showed that every computably enumerable set can be realized as an exponential Diophantine set \cite{Davis_1961}.\footnote{With a time machine and a copy of Green and Tao's 2008 paper \cite{Green_2008}, this could be pushed back to 1959, when Davis and Putnam proved this under the assumption of the existence of arbitrarily long arithmetic sequences of prime numbers.} If we allow ourselves to replace polynomials with arbitrary computable functions, these two results could have been written down at some point in the 50s, when the field of effective descriptive set theory was born out of the work of Kleene, Mostowski, and Addison. And finally if we just allow $p$ to be an arbitrary continuous function, then \cref{thm:polynomial-Vitali} was essentially written down in 1938 by G\"odel in \cite{G_del_1938}. So it has been known for somewhere between 40 and 90 years that non-measurable objects can consistently be quite explicit, but despite how old these results are, it seems to me that not many mathematicians are aware of just how explicit these explicit objects can be.

Most mathematicians are aware of some of the more famous independence results, but aside from possibly a broad sketch of the proof of G\"odel's incompleteness theorem and maybe the word `forcing,' very few know much about how you actually \emph{prove} that you can't prove something one way or another from some given axioms. Rather than just give the proofs of Theorems~\ref{thm:polynomial-Vitali}, \ref{thm:polynomial-inacc}, and \ref{thm:Polynomial-Banach-Tarski}, I'm going to try to tell the story of how one would even begin to prove such things, in part because the method of proof is entirely different from G\"odelian self-reference and set-theoretic forcing, the two most well-known techniques for proving independence results, and in part because Theorems~\ref{thm:polynomial-Vitali}, \ref{thm:polynomial-inacc}, and \ref{thm:Polynomial-Banach-Tarski} draw on several different important ideas from mathematical logic in the 20th century. Sections~\ref{sec:indep-results}-\ref{sec:construction} of this paper essentially follow the events described in the previous paragraph in roughly chronological order. \cref{sec:indep-results} discusses independence results generally and then describes G\"odel's inner model $L$, one of the core ideas in this paper and in set theory more broadly, in which one is able to well-order the reals in an `explicit' way. \cref{sec:eff-desc-set-thy} first deals with descriptive set theory, the study of the complexity of sets of real numbers that came out of early work in measure theory. \cref{sec:eff-desc-set-thy} then discusses effective descriptive set theory, an extension of descriptive set theory that highlights analogies and connections between it and computability theory. \cref{sec:explicitly-choosing} connects the ideas of the first two sections and in doing so allows us to upgrade our `explicit' well-ordering of $\Rb$ to a `computable' well-ordering of $\Rb$. \cref{sec:construction} brings this all together with the famous negative resolution of Hilbert's tenth problem in order to yield our main results, before discussing what would be involved in actually writing down the relevant polynomials (which we do not do).

The other side of the story---i.e., how it can be that objects like the above $f(x)$, $g(x,y)$, and $h(m,x,y,z)$ exist when in practice one seems to `need the axiom of choice' to build non-measurable functions or sets---appears to some extent in the earlier sections but comes into focus in \cref{sec:obviously}. A fair amount of work has been done by set theorists to explain this, such as Solovay's famous construction (from an inaccessible cardinal) of a model of $\ZF + \DC$ in which all subsets of $\Rb$ are measurable \cite{Solovay1970}. In \cref{sec:obviously} we discuss this from the point of view of functions of a similar form to those appearing in Theorems~\ref{thm:polynomial-Vitali}, \ref{thm:polynomial-inacc}, and \ref{thm:Polynomial-Banach-Tarski}. Using work of Suslin and Lusin from classical descriptive set theory and work of van den Dries, Knight, Pillay, and Steinhorn on model-theoretic o-minimality, we show that many similarly defined functions are provably measurable in \ZFC. We then cite Shelah and Woodin's proof that that `large cardinals imply that every reasonably definable set of reals is Lebesgue measurable' in \cite{Shelah1990} to show that in the presence of sufficiently large large cardinals, anything like Theorems~\ref{thm:polynomial-Vitali} and \ref{thm:Polynomial-Banach-Tarski} becomes impossible in that the relevant functions are now provably measurable.

Most of the notation we'll use here is standard. To avoid collision with the notation for open intervals, we'll use angle brackets for ordered tuples such as $\langle 1,2 \rangle$. We will also be using the logician's convention of representing tuples of indexed variables with a bar. For instance, $\xbar$ might represent $\langle x_1,x_2,x_3 \rangle$. In some places these notations will be used together. For instance, if $\xbar = \langle x_1,x_2,x_3\rangle$, then $\langle \xbar,n \rangle$ is the same thing as $\langle x_1,x_2,x_3,n \rangle$. We'll be representing projection maps with $\pi$, usually with a subscript. For example, an expression like $\pi_4$ is meant to represent a projection from some product $X \times Y \times Z \times W$ down to $X \times Y \times Z$, but this will be spelled out explicitly when it is used. Finally, we are taking $0$ to be a natural number and all suprema and infima are understood to be computed in the extended real numbers, which we will denote by $\olR$. %

I would like to thank Sean Cody, Elliot Glazer, Jem Lord, Jayde Massmann, Yuval Paz, and Connor Watson for helping me to fake a working understanding of modern set theory. I would also like to thank Michael Barz, Matt Broe, Kene Ezi, Huang Xu, Connor Lane, Shyam Ravishankar, and many others for helping me to pretend to be somewhat fluent in other various parts of mathematics (e.g., category theory, Galois theory, and real analysis). 

%

%

%
%

%
%

%

%
%


\section{Independence results and G\"odel's $L$}
\label{sec:indep-results}

A major aspect of set theory research is independence results, proving that certain statements can neither be proven nor refuted using certain axiomatic systems. Of course one of the most famous theorems of the 20th century, G\"odel's incompleteness theorem, is of this form.\footnote{Although the incompleteness theorems are not uniquely set-theoretic in nature, as they apply equally well to Peano arithmetic and other relatively weak formal systems.} In some sense, though, it is a little unfortunate that the incompleteness theorem is the most famous independence result, as it is a fairly subtle result to interpret semantically (involving models with non-standard natural numbers which `appear finite' from within the model but are externally infinite and may code infinitely long, secretly unsound proofs). Set-theoretic forcing, originally introduced by Cohen to establish the independence of the axiom of choice from \ZF\ and the continuum hypothesis from \ZFC, is likewise somewhat difficult to understand, even given the well-documented relationship between some aspects of forcing and the internal logic of certain kinds of sheaf toposes. %
The idea of `generating an entirely new real number out of thin air' is a little hard to get accustomed to (as is the idea of the `space of surjections from $\Nb$ onto $\Rb$ which has no points but is nevertheless non-empty').

Given the fact that G\"odelian incompleteness and forcing are the two most famous varieties of independence proofs, it's reasonable that these kinds of results have a certain mystique among most mathematicians. In my mind however, this is something of a shame because there are less well-known set-theoretic independence proofs that are, I would argue, more conceptually accessible than the two most famous ones. In fact, on some level, there are many `independence phenomena' that are far more familiar to most mathematicians but are not usually conceived of as `independence results' per se. Consider, for example, the following suggestively phrased proposition.

\begin{prop}\label{prop:Abel-indep}
  The abelianity hypothesis is independent of $\mathsf{GRP}$. %
\end{prop}
\begin{proof}
  $\Zb/6\Zb$ is a model of $\mathsf{GRP}$ that satisfies the abelianity hypothesis and $S_3$ is a model that does not. Therefore, by the soundness theorem for first-order logic, there can be no proof of either the abelianity hypothesis or its negation from the axioms of $\mathsf{GRP}$. %
\end{proof}

With the majority of set-theoretic independence phenomena, the very basic core argument is the same. %
One exhibits two models of something like \ZF\ or \ZFC, a model satisfying a given statement and another model failing to satisfy that statement. The primary difference is that models of \ZFC\ are large and rich enough to encode the majority of mathematics.

While models of set theory are obviously far harder to visualize than, say, groups with $6$ elements, sometimes the relative construction of these models can be visualized fairly clearly. Like many mathematicians, set theorists' intuition is often founded on cartoonishly simple mental pictures, such as the standard drawing of a model of \ZFC\ (see Figure~\ref{fig:ZFC-model}). A model of \ZFC\ is somewhat like a great tree, rooted at the empty set and with the ordinals forming a transfinitely tall trunk off of which the various sets coding familiar mathematical objects branch. This picture immediately gives two coordinates which describe the broad structure of the majority of constructions of one model of $\ZF(\mathsf{C})$ from another: One can imagine a model extending another in height, leading to the idea of large cardinals, or one can imagine a model that extends another in width, as is the case in set-theoretic forcing, or a model which is the result of pruning another to some narrower width, which is the very basic idea of \emph{inner models} such as G\"odel's $L$. Many more elaborate proofs involve performing several (or even infinitely many) of these constructions in succession, such as in Cohen's original proof of the independence of choice from \ZF, which involves forcing to widen a model and then pruning down to a narrower symmetric submodel (a kind of inner model specialized for the context of forcing) to kill the axiom of choice by carefully cutting certain choice functions out of the model.
\begin{figure}
  \centering
  \ZFCModelFigure
  \caption{A model $M$ of $\mathsf{ZF}(\mathsf{C})$ (not to scale).} %
  \label{fig:ZFC-model}
\end{figure}

\subsection{The easiest independence result}

The easiest of these pictures to explain in a moderately honest way is shortening. Cutting a model of \ZFC\ off at some arbitrary height will typically not result in another model of \ZFC\ (for instance if we chop the model in Figure~\ref{fig:ZFC-model} off between $\Rb$ and $2^\Rb$, the resulting structure fails to satisfy the power set axiom since there will literally be no sets of subsets of $\Rb$ as elements left). In order for the remainder to be a model of \ZFC, the height at which we cut needs to be special, specifically being what's called a `wordly' cardinal. For the sake of continuity with the rest of this paper though, we'll focus on the stronger condition of \emph{inaccessibility}.

Recall that two sets $A$ and $B$ have the same \emph{cardinality} if there is a bijection between them. The axiom of choice is equivalent to the statement that every set has the same cardinality as some ordinal. Since the ordinals are well-ordered, there is a unique least ordinal of any given cardinality, which in set theory is conventionally identified with the cardinality itself. For instance, $\aleph_0$, the cardinality of countable sets, is represented by the first infinite ordinal, and $\aleph_1$, the smallest uncountable cardinal, is the first uncountable ordinal (i.e., the first ordinal that does not admit a bijection with $\Nb$).\footnote{We should note that there is distinct notation in set theory for ordinals and cardinals. For instance, $\aleph_1$ is meant to indicate the least uncountable cardinal and $\omega_1$ is meant to indicate the least uncountable ordinal, even though officially they're the same object. This distinction matters in terms of signaling intentionality but also in the interpretation of arithmetic operations. $\aleph_1 + 1$ is $\aleph_1$, since addition here is being interpreted in the sense of cardinality, but $\omega_1 < \omega_1+1$, because addition here is interpreted as addition of ordinals. To keep the exposition in this paper light, however, we will not be using any ordinal-specific notation.} The cardinality of a set $A$ (i.e., the smallest ordinal admitting a bijection with $A$) is typically written $|A|$.

\begin{defn}
  An uncountable cardinal $\kappa$ is \emph{inaccessible} if $2^\lambda < \kappa$ for any $\lambda < \kappa$ and for any family $(X_i)_{i \in I}$ of sets with $|X_i| < \kappa$ for each $i$ and $|I| < \kappa$, we have that $|\bigcup_{i \in I}X_i| < \kappa$.
\end{defn}

Inaccessibility of $\kappa$ is equivalent to saying that the category of sets of cardinality less than $\kappa$ is (equivalent to) a Grothendieck universe (i.e., a small category which resembles the full category of sets strongly). Note that the two conditions listed essentially say that the collection of sets smaller than $\kappa$ satisfies the power set axiom (since $\lambda < \kappa$ entails $ 2^\lambda < \kappa$) and a strong form of replacement (since any coproduct of sets of size smaller than $\kappa$ indexed by a set of size smaller than $\kappa$ is a set of size smaller than $\kappa$). Power set and replacement are really the strongest axioms of \ZFC\ and are generally the hardest to arrange (especially together) when trying to build a model.

\begin{defn}
  Given a model $M$ of \ZFC\ and a cardinal $\kappa$ in $M$, a set $x$ is \emph{hereditarily of cardinality less than $\kappa$} if $x$ has fewer than $\kappa$ elements, every element of $x$ has fewer than $\kappa$ elements, every element of every element of $x$ has fewer than $\kappa$ elements, every element of every element of every element of $x$ has fewer than $\kappa$ elements, and so on. $H_\kappa^M$ is the set of elements of $M$ that are hereditarily of cardinality less than $\kappa$.  We may drop the superscript $M$ and just write $H_\kappa$ if $M$ is clear from context.
\end{defn}

For any inaccessible $\kappa \in M$, $H_\kappa$ 
is a model of \ZFC\ (see Figure~\ref{fig:inacc-card}). %
This gives us what is probably the easiest \ZFC\ independence result.

\begin{figure}
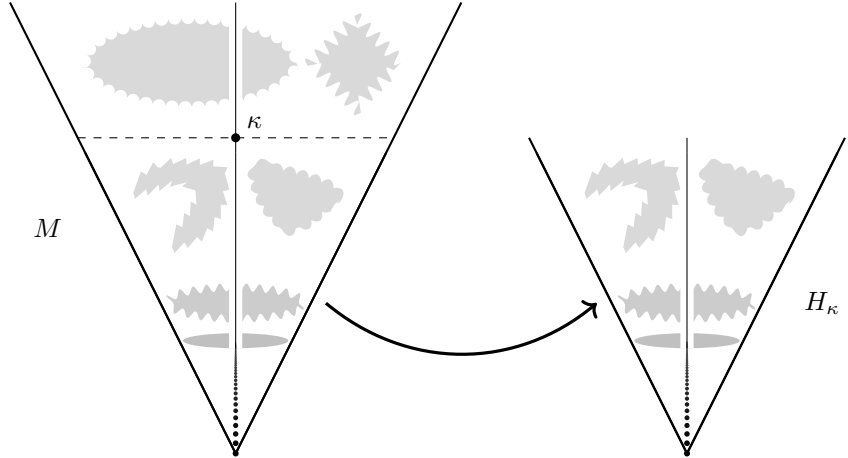

  \centering
  \InaccFigure
  \caption{An inaccessible cardinal $\kappa$ in a model $M$ of \ZFC\ and the corresponding initial segment $H_\kappa$, which is itself a model of \ZFC.}
  \label{fig:inacc-card}
\end{figure}

\begin{prop}\label{prop:basic-inacc-indep}
  If there is a model of \ZFC\ containing an inaccessible cardinal,\footnote{We needed to include this caveat because, strictly speaking, we don't know that \ZFC\ is consistent with the existence of an inaccessible cardinal. We also don't strictly speaking \emph{know} that \ZFC\ or even weaker theories are consistent. This is an important fundamental fact about metamathematics, but I do not want to dwell on it too much in this paper as it slides pretty directly into philosophical questions about the ontology of mathematics that I would guess most mathematicians don't like thinking about, and I am trying to emphasize the fact that the majority of set-theoretic independence results are more similar in spirit to \cref{prop:Abel-indep} than to G\"odelian incompleteness. %
  } then the statement ``there is an inaccessible cardinal'' is independent of \ZFC.
\end{prop}
\begin{proof}
  Fix a model $M$ of \ZFC\ with an inaccessible cardinal. Let $\kappa$ be the least inaccessible cardinal in $M$ (which exists by the fact that the cardinals are well-ordered). $H_\kappa^M$ is now a model of \ZFC\ such that no set in $H_\kappa^M$ is itself a Grothendieck universe. Therefore $H_\kappa^M$ is a model of \ZFC\ satisfying ``there are no inaccessible cardinals.'' Since we now have a model of \ZFC\ with an inaccessible cardinal and a model of \ZFC\ without an inaccessible cardinal, the statement ``there is an inaccessible cardinal'' must be independent of \ZFC.
\end{proof}

\subsection{The narrowest inner model}
In what is arguably the first major result of set-theoretic metamathematics, G\"odel famously showed that the consistency of \ZF\ implies the consistency of \ZFC\ and the (generalized) continuum hypothesis. He did this by building the first example of an inner model.%

\begin{defn}
  Given a model $M$ of \ZF, a class $N \subseteq M$ is an \emph{inner model} if $N$ contains all of the ordinals in $M$, $N$ is \emph{transitive} (i.e., satisfies that if $x \in N$ and $y \in x$, then $y \in N$), and $N$ is a model of \ZF.
\end{defn}

G\"odel showed how to build the `narrowest' inner model $L$, which can be thought of as the model of \ZF\ `generated' by the ordinals (see Figure~\ref{fig:Godels-L}), although this description is hard to make precise in a robust way. %
When it's important to indicate that we are thinking about $L$ inside a particular model, 
we typically denote this with a superscript (i.e., $L^M$ is $L$ built in $M$).

\begin{fact}[G\"odel] $ $\label{fact:Godel-L} There is a first-order formula $\varphi_L(x)$ in the language of set theory such that in any model $M$ of \ZF, $L^M \coloneq \{a \in M : M~\text{satisfies}~\varphi_L(a)\}$ is an inner model of \ZF\ satisfying the axiom of choice and the generalized continuum hypothesis. Moreover, for any inner model $N \subseteq M$, $L^N = L^M$, so in particular $L^M \subseteq N$.
\end{fact}

It is tempting to try to \emph{define} $L^M$ as the intersection of all inner models of $M$, which would immediately give a notion of `the inner model generated by a class $C \subseteq M$.' This does work for singletons (i.e., there is always a smallest inner mode $L(x)$ satisfying that $\{x\} \subseteq L(x)$), but this can fail in the case of infinite sets: %
There can be a set $x$ such that the intersection of all inner models $M$ satisfying $x \subseteq M$ is not an inner model \cite{Blass1981}. Moreover this is not really a usable definition of $L(x)$ since in general one knows very little about the collection of all inner models of a model of \ZF.

  \begin{figure}
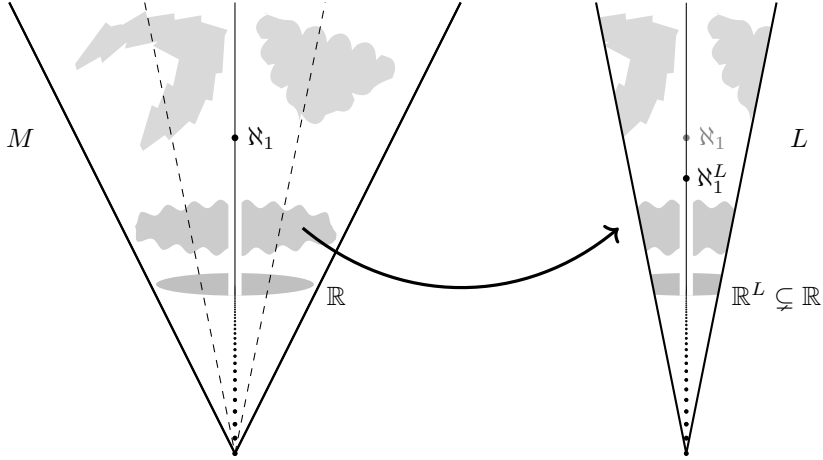

    \centering
    \GodelLFigure
    \caption{G\"odel's inner model $L$ for a typical model $M$ of \ZFC. $\Rb^L$ will usually be a strict subset of $\Rb$. Likewise, $\aleph_1$, the first uncountable ordinal, will in general be smaller in $L$ than it is in $M$.}
  \label{fig:Godels-L}
\end{figure}

The relationship between a model $M$ of \ZF\ and its corresponding $L^M$ is analogous to the relationship between a group $G$ and its center $Z(G)$ in a few ways: It depends on the initial model (i.e., for non-isomorphic models $M$ and $N$ of \ZF, $L^M$ and $L^N$ will not in general be isomorphic\footnote{Although for well-founded models, which set theorists primarily consider, it will always be the case that one of $L^M$ and $L^N$ is an initial segment of the other.  The focus on well-founded models in set theory is analogous to the focus on Noetherian rings in algebra, as they are both significantly easier to work with and fairly natural and interesting in their own right.}), it is idempotent (i.e., $L^{L^M} = L^M$), and in general $L^M$ is more `controlled' than $M$ (e.g., $L^M$ satisfies the axiom of choice even when $M$ does not, like how $Z(G)$ is abelian even if $G$ is not). $L^M$ is also first-order definable in $M$, like how $Z(G)$ is first-order definable in $G$, although, of course, $\varphi_L(x)$ is significantly more complicated than $\varphi_Z(x) \equiv \forall y (x\cdot y = y\cdot x)$.

Another rough analogy one might make is to compare the relationship between a typical model $M$ of \ZFC\ and $L$ to the relationship between the complex numbers $\Cb$ and the algebraic numbers $\overline{\Qb}$. $L$ in some sense represents the core of `explicitly definable' elements of $M$ (although there are subtle issues with making this statement precise \cite{Hamkins2024}). Moreover, it can easily be the case that $\Rb^L = \Rb \cap L$ is a proper subset of $\Rb$ or even is countable as a set in $M$.

The fact that $\Rb^L$ can be countable (in $M$) is at least reminiscent of the fact that $\overline{\Qb}$ is countable, although on some literal level the comparison is misleading: $L$ contains essentially every explicitly named real number (such as $\pi$, $e$, $\gamma$, etc.) as these can be proven to exist in \ZFC, which $L$ satisfies. A closer analogy would be to say that $\Rb^L$ is like the field of computable real numbers, although this is again not literally true, since \ZFC\ proves the existence of many specific non-computable real numbers (such as the halting oracle $0'$ or Chaitin's constant $\Omega$). Using a generalized notion of computation though, this can actually be made into a precise statement. $\Rb^L$ is the set of real numbers that can be computed by `ordinal Turing machines,' which are a kind of infinitary generalization of Turing machines that are intimately related to $L$ \cite{Koepke2005}.

\subsection{The definable well-ordering of $L$}
\label{sec:definable-well-ord}

One particularly important aspect of the fine structure of $L$ is that there is a precise mechanism that results in the axiom of choice holding, unlike in arbitrary models of \ZFC. One way to understand this mechanism is that in the construction of $L$, sets are added `carefully' so that it is possible to assign an ordinal `serial number' to each set. In other words, there is a bijection $f : \Ord \to L$ that is definable in $L$. This induces a definable global well-ordering $<_L$ of $L$, which then gives a definable global choice function: Whenever you need to pick an element of a set $X$, you can just take the $<_L$-least element of $X$ (which always exists since $<_L$ is a well-ordering). Another way to phrase this is that we are choosing the `first element of $X$ that showed up when we were building $L$.'

One thing to note is that, while $<_L$ is `canonical' in the sense of being specifically definable, it is not `canonical' in the sense of satisfying any more specific `naturality' properties. If $ \Lambda_{24} <_L \pi + e$, this doesn't mean there's any special relationship between the Leech lattice and the question of whether $\pi+e$ is an irrational number. Likewise there are arbitrary choices that go into the specific definition of $<_L$, so it could easily be the case that either $\ell^2 <_L \Qb_{19}$ or $\Qb_{19} <_L \ell^2$ (depending on whether we like Hilbert spaces or the $p$-adic numbers more). Perhaps more pertinently to ordinary mathematical concerns, if we use $<_L$ to build objects one typically builds with the axiom of choice (e.g., maximal ideals, bases of vector spaces, and extensions of linear operators on Banach subspaces), there's no reason to expect anything special about the resulting objects, aside from the fact that they're explicitly definable. %

For subsets of Polish spaces in particular, the definability of $<_L$ has the consequence that these kinds of choicy objects can be relatively simple to define in the sense of descriptive set theory. Specifically, G\"odel showed\footnote{Although in his notes, he credited the observation that the definition of the well-ordering can be phrased in such a simple way to Ulam \cite[p.~197]{Kreisel1980-qk}. (See also \cite[p.~151]{Kanamori2003}.)} that in $L$ the projection of the complement of the projection of a Borel set may fail to be measurable.

\begin{fact}[G\"odel]\label{fact:wo-defbl-first}
  Assuming $\VV = \LL$,\footnote{$\VV=\LL$ is the set-theoretic axiom that states `all sets are in $\LL$.' $V$ is the notation for the class of all sets in a model of \ZF. In other words, $M$ satisfies $V = L$ if and only if $M = L^M$.%
} there is a $G_\delta$ set (i.e., a countable intersection of open sets) $A \subseteq \Rb^4$ such that
    \[
      \pi_3\left( \Rb^3 \setminus \pi_4(A) \right) = \{\langle x,y \rangle \in \Rb^2 : x <_L y\},%
    \]
    where $\pi_4 : \Rb^4 \to \Rb^3$ and $\pi_3 : \Rb^3 \to \Rb^2$ are projection maps.%
  \end{fact}
  Moreover, it can be arranged\footnote{This is actually automatic for the definition of $<_L$ one typically writes down because every real number in $L$ is built at a countable stage of the transfinite construction of $L$.} that $\{x \in \Rb : x <_L r\}$ is countable and $\{x \in \Rb : r <_L x\}$ is co-countable for every $r \in \Rb$, implying that $\pi_3\left( \Rb^3 \setminus \pi_4(A) \right)$ is not Lebesgue measurable by Fubini's theorem:
  \[
    \int_0^1 \int_0^1
    \begin{rcases}
      \begin{dcases}
        1 & x <_L y \\
        0& x \not<_L y
      \end{dcases}
    \end{rcases}
    dx dy = 0 \neq 1 = \int_0^1 \int_0^1
    \begin{rcases}
      \begin{dcases}
        1 & x <_L y \\
        0& x \not<_L y
      \end{dcases}
    \end{rcases}
    dy dx.
  \]

  This gives us an immediate corollary (together with the basic results of measure theory that $\Rb^2$ and $\Rb$ with their standard Lebesgue measures are isomorphic as measure spaces) in the same manner as our other independence proofs.

  \begin{cor}\label{cor:L-non-meas}
    \ZFC\ does not prove that for every Borel set $B \subseteq \Rb^3$, $\pi_2(\Rb^2 \setminus \pi_3(B))$ is Lebesgue measurable, where $\pi_3 : \Rb^3 \to \Rb^2$ and $\pi_2 : \Rb^2 \to \Rb$ are projection maps.
  \end{cor}

  Of course at the moment, we don't have a true independence result because we don't know that it is consistent with \ZFC\ that all such sets are Lebesgue measurable. In fact, it is relatively hard to prove that there can be models of \ZF\ that don't satisfy $V = L$ in the first place, as this was an open problem until the development of forcing by Cohen.

Given \cref{fact:wo-defbl-first}, one might naturally wonder how `complicated' the set $A$ needs to be. Can it in some sense `actually be written down'? Or is the proof of its existence from $\VV = \LL$ somehow essentially non-constructive? Ordinary descriptive set theory does not really distinguish between different $G_\delta$ sets. They are all `equally complicated' in that they have the same Borel rank. In order to meaningfully and precisely answer these questions about $A$, we need to use ideas from computability theory. 
%

%
%
%


\section{Effectively describing sets}
\label{sec:eff-desc-set-thy}

Descriptive set theory as a field starts by asking the question `What can we say about regularity properties of subsets of $\Rb$ (and other metric spaces) that can be built using topologically tame sets, such as open and closed sets, using basic operations, such as (countable) unions and intersections, complements, and projections?' It has its roots in the beginning of measure theory and the work of Cantor, Borel, Lebesgue, Suslin, Lusin, and Sierpi\'nski. %
Despite the fact that questions about measurability of simple-to-define sets are now largely settled, descriptive set theory continues to be an active area of research, with many applications in various areas of math (both within and outside of mathematical logic). It's often useful in showing that classifying some family of objects up to isomorphism via `simply definable' invariants is impossible.%

One of the central concepts in descriptive set theory is that of a \emph{pointclass}. It is mildly onerous to give a precise definition of the word `pointclass' that will remain correct as we start considering effective descriptive set theory, so for the purposes of this paper, we will just take the word to refer to the concepts given in Definitions~\ref{defn:proj}, \ref{defn:semicomp}, \ref{defn:eff-proj}, and \ref{defn:eff-Borel}.

The most familiar examples of pointclasses are probably the classes of open and closed sets, but the classes of $G_\delta$ and $F_\sigma$ sets are also fairly common. Beyond $G_\delta$ and $F_\sigma$ is of course the pointclass of all Borel sets, and even further beyond this is the \emph{projective hierarchy} (see Figure~\ref{fig:projective-strict}), which is built out of Borel sets by successive applications of the operations of projections and complements and is a major topic of descriptive set theory and its interaction with set theory.

$G_\delta$ and $F_\sigma$ are remnants of an older notation that was abandoned by descriptive set theorists as it does not scale well to more complicated sets. %
The following definition features some of the notation adopted by descriptive set theorists. 

Recall that a \emph{Polish space} is a completely metrizable separable space. For the moment, kindly ignore the heavy choice of notation in \cref{defn:proj} (i.e., the bold font and under-tildes); the rationale for this will become clear after \cref{defn:eff-proj}.

\begin{defn}\label{defn:proj}
  The \emph{(boldface) projective hierarchy} is a family of pointclasses of subsets of Polish spaces defined inductively as follows. %
  \begin{itemize}
  \item A subset $A$ of a Polish space $X$ is $\bSigma^1_1$ (also called \emph{analytic}) if there is a Borel set $B \subseteq X \times \Nb^\Nb$ such that $A = \pi(B)$.
  \item A set is $\bPi^1_n$ if it is the complement of a $\bSigma^1_n$ set.
  \item A set $A \subseteq X$ is $\bSigma^1_{n+1}$ if there is a $\bPi^1_n$ set $B \subseteq X \times \Nb^\Nb$ such that $A = \pi(B)$.
  \end{itemize}
  Finally, a set is $\bDelta^1_n$ if it is both $\bSigma^1_n$ and $\bPi^1_n$.
\end{defn}

Suslin's theorem (later generalized by Lusin's separation theorem) is a particularly important structural fact about the projective hierarchy.

\begin{fact}[Suslin]\label{fact:Suslin}
  A set is $\bDelta^1_1$ if and only if it is Borel.
\end{fact}

\begin{figure}
  \centering
\begin{tikzcd}
	& {\bSigma^1_1} && {\bSigma^1_2} && {\bSigma^1_3} & { } \\
	{\text{Borel} = \bDelta^1_1} && {\bDelta^1_2} && {\bDelta^1_3} & \cdots & {} \\
	& {\bPi^1_1} && {\bPi^1_2} && {\bPi^1_3} & { }
	\arrow["\subsetneq"{marking, allow upside down}, draw=none, from=1-2, to=2-3]
	\arrow["\subsetneq"{marking, allow upside down}, draw=none, from=1-4, to=2-5]
	\arrow["\cdots"{marking, allow upside down}, draw=none, from=1-6, to=1-7]
	\arrow["\subsetneq"{marking, allow upside down}, draw=none, from=2-1, to=1-2]
	\arrow["\subsetneq"{marking, allow upside down}, draw=none, from=2-1, to=3-2]
	\arrow["\subsetneq"{marking, allow upside down}, draw=none, from=2-3, to=1-4]
	\arrow["\subsetneq"{marking, allow upside down}, draw=none, from=2-3, to=3-4]
	\arrow["\subsetneq"{marking, allow upside down}, draw=none, from=2-5, to=1-6]
	\arrow["\subsetneq"{marking, allow upside down}, draw=none, from=2-5, to=3-6]
	\arrow["\subsetneq"{marking, allow upside down}, draw=none, from=3-2, to=2-3]
	\arrow["\subsetneq"{marking, allow upside down}, draw=none, from=3-4, to=2-5]
	\arrow["\cdots"{marking, allow upside down}, draw=none, from=3-6, to=3-7]
\end{tikzcd}
  \caption{For any uncountable Polish space $X$, the boldface projective hierarchy is strict. The same is true of the lightface projective hierarchy for countable or uncountable basic Polish spaces.}
  \label{fig:projective-strict}
\end{figure}
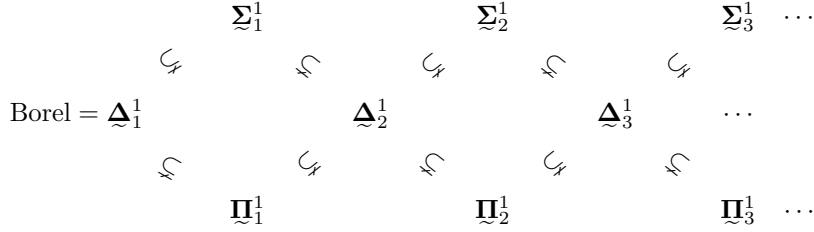

\emph{Effective} descriptive set theory grew out of a body of connections and analogies between descriptive set theory and computability theory and specifically the work of Kleene, Mostowski, and Addison in the 1950s. Within descriptive set theory more generally, it is often useful for fine-grained analysis and has applications to problems that do not overtly involve computability theory, such as work by Lutz, Qi, and Yu in \cite{Lutz2024} on the Hausdorff dimensions of Hamel bases.

Here we will review some basic definitions and facts from effective descriptive set theory which we will eventually need. We will mostly follow \cite{Moschovakis2009-je}, although we will use slightly different terminology, with the biggest difference being the uniform replacement of the word `recursive' (in the sense of recursion theory) with the word `computable' (in the sense of computability theory, which is the same thing as recursion theory). We should also highlight that the following term is not standard.\footnote{Effective descriptive set theory in general deals with `computably presentable' Polish spaces, which includes the vast majority of explicitly named metric spaces in mathematics.}

\begin{defn}
  A \emph{basic Polish space} is a finite product of the spaces $\Nb$, $\Nb^\Nb$, $\Rb$, and $\olR$ (where $\olR \coloneq [-\infty,\infty]$ is the extended real numbers).
\end{defn}

Obviously we could add more to this list if we wanted to, but these are the only spaces we'll use.

We need to fix particular enumerations of topological bases of basic Polish spaces. We will not actually be specifically referencing these enumerations at all; we've provided them more for the thematic purpose of explicitness (and for the sake of fun). As such, let $\nu_p(n)$ be the $p$-adic valuation of $n$ (i.e., the highest power of $p$ that divides $n$). Let $(p_n)_{n \in \Nb}$ be the sequence of prime numbers (i.e., $p_0 = 2$, $p_1=3$, etc.). For each of the spaces $\Nb$, $\Nb^\Nb$, $\Rb$, and $\olR$ we fix an explicit enumeration of a topological basis:
\begin{itemize}
\item $N(\Nb,n) = \{n\}$.
\item For $n > 0$,
  \[
    N(\Nb^\Nb,n-1) = \{\alpha \in \Nb^\Nb : (\forall i < \nu_2(n))\alpha(i) = \nu_{p_{i+1}}(n)\}.
  \]
\item For $n > 0$, $N(\Rb,n-1)$ is the open interval
  \[
    \left( \frac{\nu_3(n) - \nu_5(n)}{1+\nu_2(n)}, \frac{\nu_3(n) - \nu_5(n) + \nu_7(n) + 1}{1+\nu_2(n)} \right).
  \]
\item $N(\olR,3n) = N(\Rb,n)$, $N(\olR,3n+1) = (n,\infty]$, and $N(\olR,3n+2) = [-\infty,-n)$.
\item For any product $X \times Y$ of basic Polish spaces, for $n > 0$, $N(X\times Y, n-1) = N(X,\nu_2(n))\times N(Y,\nu_3(n))$.
\end{itemize}
The fact that we have many redundant entries in these enumerations doesn't really matter. Moreover, the particular choices of basic open neighborhoods do not matter that much. All that matters is that these actually are open bases and that basic facts about sets in each basis (such as inclusion of one set in another, emptiness of intersection, etc.) are computably decidable, which is easy to verify about our choices here. 

The above exercise of choosing values for $N(X,n)$ gives a good example of how computability theorists (and other logicians for that matter) think. Most mathematicians engage with the natural numbers combinatorially or in an algebraic or analytic capacity. To a computability theorist, by contrast, a natural number is often just data. Any finite, possibly heterogeneous, bundle of data can be coded as a single natural number, as in our coding of $N(\Nb^\Nb,n)$. Here we're using $\nu_2$ in a fundamentally different way than the other $p$-adic valuations, in that it codes the length of initial segment of an element of $\Nb^\Nb$ that needs to be checked to verify membership in $N(\Nb^\Nb,n)$. All of this is fairly unnatural from the point of view of number theory, but it doesn't matter here. All that matters at the moment is that we \emph{can} do it, just like with the definable well-order $<_L$ of $L$.

Recall that a set $W$ of natural numbers is \emph{computably enumerable} if, informally speaking, there is a computer program that lists the elements of $W$ in some order. %
This can be formalized of course, but all reasonable formalizations of this notion result in an equivalent theory of computation. (This is the phenomenon of \emph{Turing completeness} in computer programming.) $W$ is \emph{computable} or \emph{decidable} if there is a computer program that can decide membership of $W$. A basic exercise in computability theory is showing that $W$ is computable if and only if both $W$ and $\Nb \setminus W$ are computably enumerable. Famously Turing showed that there are computably enumerable sets that are not computable (specifically sets coding the solution to the halting problem), but there are also many examples in mathematics of sets that are known to be computably enumerable but not known to be computable, such as the set $\{|p-q| : p,q~\text{prime}\}$.

Computably enumerable sets have a familial resemblance to open sets. 
Membership in either a computably enumerable or open set is witnessed by some `finite amount of data.' This family resemblance makes the following notion fairly natural.

\begin{defn}\label{defn:semicomp}
  Given a basic Polish space $X$, a set $U \subseteq X$ is \emph{semicomputable} if there is a computably enumerable set $W$ such that $U = \bigcup_{n \in W}N(X,n)$. 
\end{defn}

It is immediate that a set $U \subseteq \Nb$ is semicomputable if and only if it is computably enumerable.

Since our goal is to analyze the complexity of projective sets in terms of computability theory, we also need the following more fine-grained notion.

\begin{defn}\label{defn:eff-proj}
  The \emph{lightface} or \emph{effective projective hierarchy} is a family of pointclasses of subsets of basic Polish spaces defined inductively as follows.
  \begin{itemize}
  \item A set $A \subseteq X$ is $\lSigma^1_1$ if there is a semicomputable set $U \subseteq X \times \Nb^\Nb$ such that $A = \pi(X \times \Nb^\Nb \setminus U)$.
  \item  A set is $\lPi^1_n$ if it is the complement of a $\lSigma^1_n$ set.
  \item A set $A \subseteq X$ is $\lSigma^1_{n+1}$ if there is a $\lPi^1_n$ set $B \subseteq X \times \Nb^\Nb$ such that $A = \pi(B)$.
  \end{itemize}
  Finally, a set is $\lDelta^1_n$ if it is both $\lSigma^1_n$ and $\lPi^1_n$.
\end{defn}

The terms `boldface' and `lightface' are in reference to the fact that historically the two families of pointclasses were distinguished by font alone. %
Since this is admittedly an awful notational convention, it is now common to further distinguish the boldface pointclasses by under-tildes.\footnote{It should be noted, however, that a lot of descriptive set theory literature only considers the boldface notion and so doesn't bother with maintaining notation for distinguishing the lightface projective sets.}

 The superscript $1$ is to distinguish from the notation for the levels of the (effective) Borel hierarchy (e.g., $\lSigma^0_1$ and $\lPi^0_1$). Since we won't be needing to think about the complexity of Borel sets in this paper, we will avoid spelling out too much of this notation, but we will still need the following definitions.

\begin{defn}\label{defn:eff-Borel}
  For any basic Polish space $X$, a set $A \subseteq X$ is a \emph{$\lSigma^0_1$ set} if it is semicomputable. A set is a \emph{$\lPi^0_1$ set} if it is the complement of a $\lSigma^0_1$ set. A set is a \emph{$\lDelta^0_1$ set} if it is both $\lSigma^0_1$ and $\lPi^0_1$.

  A set $B$ is a \emph{$\lPi^0_2$ set} if there exists a semicomputable $U \subseteq X \times \Nb$ such that
  \(
    B = \bigcap_{n \in \Nb}\{x \in X : \langle x,n \rangle \in U\}.
  \)
  A \emph{$\lSigma^0_2$ set} is the complement of a $\lPi^0_2$ set.
\end{defn}

Just as how semicomputable sets are `computably open' sets, $\lPi^0_2$ sets are `computably $G_\delta$' sets. \cref{fact:Suslin} has a generalization to the lightface projective sets, which is a bit technical to state. We will use a corollary of this a few times which is that $\lSigma^0_1$, $\lPi^0_1$, $\lSigma^0_2$, and $\lPi^0_2$ sets are all $\lDelta^1_1$.

The following basic fact will also be used frequently. The second part of this statement should be interpreted as saying that the lightface projective pointclasses are closed under `computable countable unions' and `computable countable intersections' and the boldface projective pointclasses are closed under countable unions and intersections. %

\begin{fact}[{\cite[Cor.~3E.2]{Moschovakis2009-je}}]\label{fact:Bool}$ $
  Both of the classes $\lSigma^1_n$ and $\bSigma^1_n$ are closed under finite intersections and unions. %
  If $A \subseteq X \times \Nb$ is $\lSigma^1_n$ (resp.~$\bSigma^1_n$), then
  \(
    \bigcup_{k \in \Nb} \{x \in X : \langle x,k \rangle \in A\}
  \)
  and
  \(
    \bigcap_{k \in \Nb} \{x \in X : \langle x,k \rangle \in A\}
  \)
  are $\lSigma^1_n$ (resp.~$\bSigma^1_n$) as well.
\end{fact}
Note that \cref{fact:Bool} immediately implies analogous results for the (boldface and lightface) $\Pi$ and $\Delta$ pointclasses.

It is also useful to define a notion of definability for functions, analogously to the preceding notions of definability for sets.

\begin{defn}\label{defn:pointclass-fun}
  Given two basic Polish spaces $X$ and $Y$ and a pointclass $\Gamma$, a function $f : X \to Y$ is a \emph{$\Gamma$ function}\footnote{In \cite{Moschovakis2009-je}, these are also sometimes referred to as \emph{$\Gamma$-recursive} functions.} if the set
  \(
    \{\langle x,n \rangle \in X \times \Nb : f(x) \in N(Y,n)\}
  \)
  is in $\Gamma$.
\end{defn}

This generalizes a few other notions: A function $f : X \to Y$ is
\begin{itemize}
\item $\bDelta^1_1$ if and only if it is Borel,%
\item $\bSigma^0_1$ if and only if it is continuous (where $\bSigma^0_1$ is the pointclass of open sets), and
\item $\lSigma^0_1$ if and only if it is computable in the sense of computable analysis. %
\end{itemize}
Moreover a function $f : \Nb^k \to \Nb$ is computable in the sense of computability theory if and only if it is a $\lSigma^0_1$ function. The fact regarding $\bDelta^1_1$ functions follows from \cref{fact:Suslin}.

We will be using the following fact in several places.

\begin{fact}[{\cite[Thm.~3E.5 and 3E.7]{Moschovakis2009-je}}]\label{fact:function-sub}
  For any basic Polish spaces $X$ and $Y$, any $\lDelta^1_1$ function $f: X \to Y$, and any set $A \subseteq Y$, if $A$ is $\lSigma^1_n$ (resp.\ $\lPi^1_n$, $\lDelta^1_n$), then the preimage $f^{-1}(A)$ is $\lSigma^1_n$ (resp.\ $\lPi^1_n$, $\lDelta^1_n$) as well.

   For any uncountable basic Polish spaces $X$ and $Y$, there is a $\lDelta^1_1$ bijection $f : X \to Y$ with $\lDelta^1_1$ inverse.
\end{fact}

This is extremely useful because the class of $\lDelta^1_1$ functions includes all computable functions and is closed under basic operations such as composition. The second part of \cref{fact:function-sub} (as well as its boldface analog) is part of the reason why so much descriptive set theory literature is phrased explicitly in terms of Baire space, $\Nb^\Nb$, rather than $\Rb$, despite the historical origins of the field. All uncountable Polish spaces are `isomorphic' in terms of most of the structure relevant to descriptive set theory, and this fact is essentially true in the sense of effective descriptive set theory as well. There is a clear analogy here to computability theory. One naturally encounters many different countable sets in computability theory, 
such as $\Nb$, $\Nb^n$, $\Zb$, and $\Qb$, but these are all `the same' in the sense that there are computable bijections between them.

That said, there are of course meaningful differences between different uncountable Polish spaces (e.g., $\olR$ is compact, $\Rb$ is locally compact, and $\Nb^\Nb$ is not even $\sigma$-compact), but in terms of descriptive set theory we only really see these differences near the low end of the Borel hierarchy (as we will see later in \cref{sec:extra-quant-sigma-compact}). For instance, we cannot replace $\lPi^0_2$ with $\lPi^0_1$
in the following fact. %

\begin{fact}[{\cite[Ex.~3E.12]{Moschovakis2009-je}}]\label{fact:effective-Suslin}
  For any basic Polish space $X$, a set $A \subseteq X$ is $\lSigma^1_1$ if and only if there is a $\lPi^0_2$ set $B \subseteq X \times \Rb$ such that $A = \pi(B)$.
\end{fact}

\begin{lem}\label{lem:replace-Baire}
  For any basic Polish space $X$ and $n \geq 1$, a set $A \subseteq X$ is $\lSigma^1_{n+1}$ if and only if there is a $\lPi^1_n$ set $B \subseteq X \times \Rb$ such that $A = \pi(B)$.
\end{lem}
\begin{proof}
  It follows from \cref{fact:function-sub} that there is a $\lDelta^1_1$ bijection $f : \Rb \to \Nb^\Nb$. We can now apply \cref{fact:function-sub} to the $\lDelta^1_1$ map $g : X \times \Rb \to X \times \Nb^\Nb$ defined by $g(\langle x,y \rangle) = \langle x,f(y) \rangle$ to get the required result.
\end{proof}

\Lusin\ and Sierpi\'nski independently showed that one can build `universal $\bSigma^1_n$ sets.' In other words, for any uncountable Polish space $X$, there is a $\bSigma^1_n$ set $A \subseteq X\times X$ such that for every $\bSigma^1_n$ set $W \subseteq X$, there is an $e \in X$ such that $\{x \in X : \langle x,e \rangle \in A\} = W$. This can be used to show that the (boldface and lightface) projective hierarchies are strict (see Figure~\ref{fig:projective-strict}) via a diagonalization argument that is reminiscent of Turing's proof of the unsolvability of the halting problem (which we can rephrase in our language here as the statement that there is a $\lSigma^0_1$ subset of $\Nb$ that is not $\lDelta^0_1$).

Effective descriptive set theory had not been developed yet at the time of \Lusin\ and Sierpi\'nski's result, but the proof of the result is so explicit that the resulting universal set is actually lightface projective, even though it indexes all boldface sets of the same complexity. %

\begin{prop}\label{prop:sigma-universal}
  For any $n$ and any basic Polish space $X$, there is a $\lSigma^1_n$ set $A \subseteq X \times \Rb$ such that for every $\bSigma^1_n$ set $B \subseteq X$, there is an $r \in \Rb$ such that
  \(
    B = \{x \in X : \langle x,r \rangle \in A\}.
  \)
\end{prop}
\begin{proof}
  This follows from \cite[Thm.~12.7]{Kanamori2003} and \cref{fact:function-sub}.
\end{proof}

\begin{defn}
  A set satisfying the conclusion of \cref{prop:sigma-universal} is called a \emph{universal $\bSigma^1_n$ set}.  \emph{Universal $\bPi^1_n$ sets} are defined similarly.
\end{defn}

Note that the complement of any universal $\bSigma^1_n$ set is a universal $\bPi^1_n$ set.


\section{Explicitly choosing things}
\label{sec:explicitly-choosing}

Now that we have introduced some of the basic ideas of effective descriptive set theory, we are in a position to answer the question posed at the end of \cref{sec:indep-results}. Just as with the universal $\bSigma^1_n$ sets in \cref{prop:sigma-universal}, G\"odel's definition is explicit enough that it is possible to show that it is actually $\lDelta^1_2$ rather than just $\bDelta^1_2$.

\begin{fact}[{G\"odel \cite[Thm.~13.9]{Kanamori2003}}]\label{fact:wo-defbl}
  The $G_\delta$ set in \cref{fact:wo-defbl-first} can be chosen to be $\lPi^0_2$. %
\end{fact}

\cref{fact:function-sub} immediately gives the following corollary.


%

%

\begin{cor}
  Assuming $\VV= \LL$, for every basic Polish space $X$, there is a $\lDelta^1_2$ well-ordering $<_{\LL}^X$ of $X$.
\end{cor}
\begin{proof}
  If $X$ is countable, then there is clearly a computable (i.e., $\lDelta^0_1$) well-ordering, which is therefore also $\lDelta^1_2$.

  If $X = \Rb$, Facts~\ref{fact:effective-Suslin} and \ref{fact:wo-defbl} and \cref{lem:replace-Baire} establish that $<_L$ restricted to $\Rb^2$ is $\lSigma^1_2$. To see that it is also $\Pi^1_2$, note that $x \not <_L y$ if and only if $x = y$ or $y <_L x$. $x=y$ is a $\lPi^0_1$ (and therefore $\lSigma^1_2$) condition, so $x \not <_L y$ is $\lSigma^1_2$ by \cref{fact:Bool}.

  For other uncountable $X$, the result follows from \cref{fact:function-sub}. 
\end{proof}

Note that for powers of $\Rb$ in particular, we could also use the lexicographic order induced by $<_L$ restricted to $\Rb$. For example, we could take $\langle x,y \rangle <_L^{\Rb^2}\langle a,b \rangle$ to mean $x <_L a \vee (x = a \wedge y <_L b)$. The specifics do not matter too much. %

Given a low complexity well-ordering on a Polish space, we now have a low complexity choice function on subsets thereof and can use this to produce low complexity versions of various choicy constructions. Recall that given an equivalence relation $E$ on a set $A$, a \emph{transversal} is a set $B \subseteq A$ such that each equivalence class of $E$ contains exactly one element of $B$.

\begin{lem}\label{lem:transversal}
  Assuming $\VV = \LL$, for any $\lDelta^1_2$ equivalence relation $E$ on a $\lDelta^1_2$ set $B \subseteq \Rb^n$, there is a $\lPi^1_2$ transversal of $E$.
\end{lem}
\begin{proof}
  The set $\{\xbar \in \Rb^n : (\forall \ybar \in \Rb^n)[\xbar \in B \wedge (\ybar \notin B \vee \xbar <_{\LL}^{\Rb^n} \ybar \vee \xbar = \ybar \vee \neg(\xbar \mathrel{E} \ybar))]\}$ defines a transversal of $E$ and is the complement of the projection of a $\lDelta^1_2$ subset of $\Rb^n$ by \cref{fact:Bool}. By \cref{fact:function-sub}, 
  this implies that it is a $\lPi^1_2$ set.
\end{proof}

\begin{lem}\label{lem:Vitali-equiv-basic}
  $[0,1] \subseteq \Rb$ is $\lPi^0_1$ and the equivalence relation $x \mathrel{E} y$ on $[0,1]$ defined by $x-y \in \Qb$ is $\lSigma^0_2$. In particular, both are $\Delta^1_2$.
\end{lem}
\begin{proof}
  The fact that $[0,1]$ is $\lPi^0_1$ is obvious. It is also straightforward to show that the set $\{\langle x,x+q_n,n \rangle : x \in \Rb,~n \in \Nb\} \subseteq \Rb^2 \times \Nb$ (for any given reasonable enumeration $(q_n)_{n \in \mathbb{N}}$ of $\mathbb{Q}$)  is $\lPi^0_1$, which implies that the set $\{\langle x, y \rangle : x-y \in \Qb\}$ is $\lSigma^0_2$. By \cref{fact:Bool} and the fact that $[0,1]^2$ is also $\lPi^0_1$, this implies that $E$ is $\lSigma^0_2$.
\end{proof}

\begin{prop}\label{prop:Vitali-V-L}
  There is a semicomputable (and therefore open) set $U \subseteq \Rb^3 \times \Nb$ such that, assuming $\VV =\LL$,
  \[
    \Rb \setminus \pi_2\left( \Rb^2 \setminus \pi_3\left( \Rb^3 \setminus \pi_4\left( \Rb^3 \times \Nb \setminus U \right) \right) \right)
  \]
  is a $\lPi^1_2$ Vitali subset of $[0,1]$, where $\pi_4 : \Rb^3 \times \Nb \to \Rb^3$, $\pi_3 : \Rb^3 \to \Rb^2$, and $\pi_2 : \Rb^2 \to \Rb$ are projection maps.
\end{prop}
\begin{proof}
  By Lemmas \ref{lem:transversal} and \ref{lem:Vitali-equiv-basic}, we can find a $\lPi^1_2$ Vitali set $A \subseteq [0,1] \subseteq \Rb$. By \cref{lem:replace-Baire}, there is a $\lPi^1_1$ $B \subseteq \Rb^2$ such that $A$ is the complement of $\pi_2(B)$. By \cref{fact:effective-Suslin}, we can find a $\lPi^0_2$ set $C \subseteq \Rb^3$ such that $B$ is the complement of $\pi_3(C)$. By definition, there is a $\lSigma^0_1$ (i.e., semicomputable) set $U \subseteq \Rb^3 \times \Nb$ such that $C$ is the complement of $\pi_4(\Rb^3\times \Nb \setminus U)$. So putting it all together we have that $A = \Rb \setminus \pi_2(B) = \Rb \setminus \pi_2(\Rb^2 \setminus \pi_3(D)) = \Rb \setminus \pi_2(\Rb^2 \setminus \pi_3(\Rb^3 \setminus \pi_4(\Rb^3 \times \Nb \setminus U)))$.
\end{proof}

It should be noted that \cref{prop:Vitali-V-L} is constructive in the sense that there is a specific computer program enumerating the basic open neighborhoods of $U$. Moreover, the set $\Rb \setminus \pi_2\left( \Rb^2 \setminus \pi_3\left( \Rb^3 \setminus \pi_4\left( \Rb^3 \times \Nb \setminus U \right) \right) \right)$ exists in any model of \ZFC,\footnote{Or perhaps it would be more accurate to say that in any model of \ZFC\ there is a set that can be built according to those instructions.} but the next fact implies that it is independent of \ZFC\ whether it actually is a Vitali set.

\begin{fact}[{Krivine \cite{Krivine}}]\label{fact:Krivine-proj-meas}
  \ZFC\ is relatively consistent with the statement ``every (lightface) $\lSigma^1_n$ subset of $\Rb$ is Lebesgue measurable for every $n$.'' 
\end{fact}
\begin{proof}
  In \cite{Krivine}, Krivine showed that \ZFC\ is relatively consistent with the statement that all ordinal-definable sets of reals are Lebesgue measurable. $\lSigma^1_n$ sets are clearly ordinal-definable, so the statement follows.
\end{proof}

In particular, this statement in \cref{fact:Krivine-proj-meas} implies that any set defined in the same manner as the set in \cref{prop:Vitali-V-L} is measurable.

The word `lightface' is essential in this statement, however. Solovay famously showed that, assuming the existence of an inaccessible cardinal, there is a model of $\ZF + \DC$ in which all sets of reals are Lebesgue measurable. As an intermediate step in his proof, Solovay also built a model of $\ZFC$ in which all (boldface) projective sets of reals are measurable \cite{Solovay1970}. It was later shown by Shelah that it is not possible to remove the assumption of an inaccessible cardinal from this proof \cite{Shelah1984}. 
This is part of a general theme in set theory, namely that regularity properties of more complicated explicitly definable sets are intimately linked with larger scale large cardinal phenomena.

\begin{fact}[{Shelah \cite{Raisonnier1984,Shelah1984}}]\label{fact:Shelah-inaccess}
  If all (boldface) $\bSigma^1_3$ sets of reals are Lebesgue measurable, then $\aleph_1$ is an inaccessible cardinal in $\LL$. In particular, the statement ``all $\bSigma^1_3$ sets of reals are Lebesgue measurable'' implies the consistency of \ZFC. (See Figure~\ref{fig:Shelah-Solovay}.)
\end{fact}

\begin{figure}
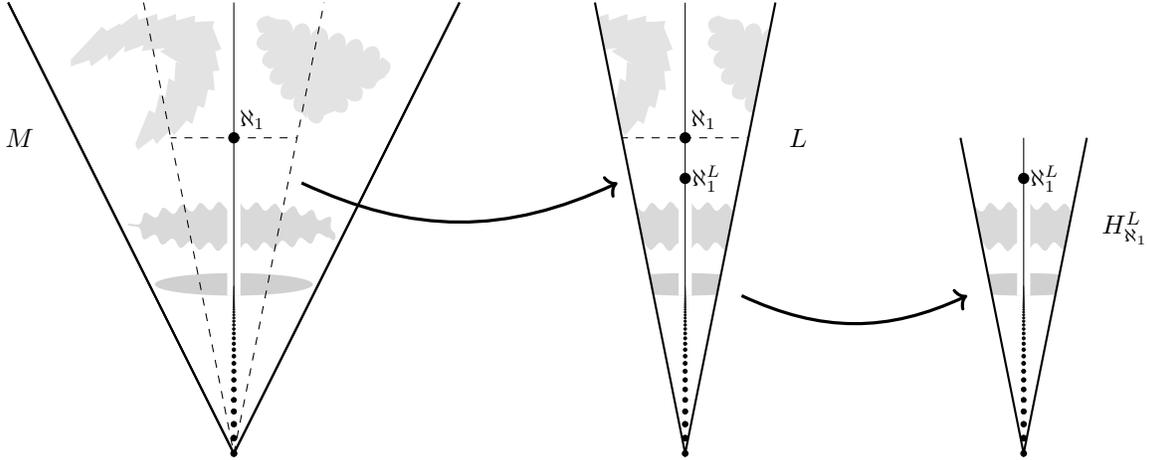

  \centering
  \SolovayInaccFigure
  \caption{If all $\protect\bSigma^1_3$ sets of reals are Lebesgue measurable, then $\aleph_1$ is inaccessible{\protect\footnotemark} in $L$, so $H_{\aleph_1}^L$ is a model of \ZFC.}
    \label{fig:Shelah-Solovay}
  \end{figure}
  \footnotetext{Note though that $\aleph_1^M$ can never be inaccessible in $M$, since the category of countable sets is not a Grothendieck universe. In the same way, $\aleph_1^L$ can never be inaccessible in $L$.}

  As noted by Shelah in \cite{Shelah1984}, \cref{fact:Shelah-inaccess} goes through in $\ZF + \AC_{\aleph_0}$ (where $\AC_{\aleph_0}$ is the axiom of countable choice). Moreover, it's known that far less than $\ZF$ itself is actually required for the proof and that it goes through in $\BZ + \AC_{\aleph_0}$ (where \BZ\ is `bounded Zermelo set theory'). 
  
  Note that \cref{fact:Shelah-inaccess} \emph{doesn't} say that if all $\bSigma^1_3$ sets of reals are Lebesgue measurable, then there is an inaccessible cardinal. This is provably impossible. A statement that only quantifies over reals and sets of reals (or other collections of uniformly small objects) cannot unconditionally imply the existence of an inaccessible cardinal because such a statement will still hold in $H_\kappa$, where $\kappa$ is the smallest inaccessible cardinal. Note that this is essentially the same argument as in the proof of \cref{prop:basic-inacc-indep}.


\section{Actually writing it down}
\label{sec:construction}
\subsection{Hilbert's tenth} 
\label{sec:Hilbert}

We will of course be needing the famous negative resolution of Hilbert's tenth problem. This is by far the best known part of this story and is already covered by many accessible expository accounts, such as \cite{Murty2019-pi}.

For the moment, we will adopt the convention of writing functions with a semicolon to indicate that certain variables are meant to be interpreted as parameters. For example, $p(\xbar;n,m)$ in the following fact is strictly speaking a polynomial in the variables $\xbar$, $n$, and $m$, but we are thinking of $\xbar$ as \emph{unknowns} and of $n$ and $m$ as \emph{parameters}. In particular, when we talk about the `number of variables' of a polynomial, were are only talking about those before the semicolon.

\begin{fact}[Matiyasevich--Robinson--Davis--Putnam]\label{fact:MRDP}
  There is a polynomial $p(\xbar;n,m)$ with integer coefficients such that for every computably enumerable set $W$ there is a natural number $m$ such that for every $n$, $n \in W$ if and only if there is a solution to $p(\xbar;n,m) = 0$ in the natural numbers.
\end{fact}

Polynomials with this property are called \emph{universal}. We will need a very slightly more specific property which is usually guaranteed in the construction of universal polynomials. To match the convention of \cite{Jones1982}, we will also allow $m$ to be a tuple of parameters rather than a single parameter.

\begin{figure}
  \centering
\begin{align*}
  \hphantom{+}&(x_1x_2x_3^2 + x_4 - (x_5-nm_3)x_6^2)^2 + (x_6 - x_5^{5^{60}})^2 + (x_7+x_6^4 - 1 - x_7 x_5^5)^2 \\
  +&(x_8 + 2m_1 - x_5^5)^2 + (x_2 - m_2 -x_9x_8)^2 + (x_1 - m_3 - x_{10}x_8)^2 + (x_{11} - x_6^{16})^2 \\
  +& ([x_3+x_1x_6^3 + x_2x_6^5 + (2(x_1-m_1x_7)(1+nx_5^5+x_3)^4+x_7 x_5^5 + x_7 x_5^5x_6^4)x_6^4][x_{11}^2-x_{11}] \\
              &\quad +[x_6^3-x_5x_2+x_2+x_8x_7 x_6^3 + (x_5^5 - 2)x_6^5][x_{11}^2-1]- x_{12})^2 \\
  +& (x_{13} - 2x_{14}x_{15}^2x_{12}^2x_{11}^2)^2 + (x_{13}^2x_{16}^2 - x_{16}^2 + 1 - x_{17}^2)^2  \\
  +& (4(x_{18}-x_{16}x_{15}x_{11}^2)^2 + x_{19} - x_{16}^2)^2 + (x_{16} - x_{12} - 1 - x_{20}x_{13} + x_{20})^2 \\
  +& (x_{21} - (x_{14}x_{11}^2 + 1)x_{12}x_{15}x_{11}^2)^2 + (x_{18} - 2x_{12} - 1 - x_{22})^2 \\
  +& (x_{23} - x_5x_{14} - x_{18}x_{21} + 2x_{18} - 4 x_{21} x_{24} + 5 x_{24})^2 \\
     +& (x_{23}^2 - (x_{21}^2 - 1) x_{18}^2 - 1)^2  + (x_{25}^2 - (x_{21}^2 - 1)x_{26}^2x_{18}^4 - 1)^2  \\
  +& ((x_{23}+x_{27}x_{25})^2 - ((x_{21}+x_{25}^2(x_{23}^2-x_{21}))^2 - 1)(2x_{12}+1+x_{28}x_{18})^2 - 1)^2
\end{align*}
\caption{Jones's universal polynomial{\protect\footnotemark}} %
  \label{fig:equation}
\end{figure}

\footnotetext{Note that \cite{Jones1982} takes the natural numbers to be the positive integers. To use Jones's polynomial with the convention $0 \in \Nb$, it would be necessary to replace each variable $x$ with $x+1$.}

\begin{defn}
  A polynomial $p(\xbar;n,\mbar)$ with integer coefficients is \emph{positively universal} if for every computably enumerable set $W$, there is a tuple of natural numbers $\mbar$ such that
  $p(\xbar;n,\mbar) \geq 0$ for any natural numbers $\xbar$ and $n$ and 
  for every $n$, $n \in W$ if and only if there is a solution to $p(\xbar;n,\mbar) = 0$ in the natural numbers.
\end{defn}
Obviously any universal polynomial $p$ yields a positively universal polynomial $p^2$ with twice the degree, but universal polynomials are typically constructed as a system of polynomials which are then combined as a sum of squares. This is almost certainly true of all known constructions of universal polynomials.

The original proof of the existence of universal polynomials, while constructive, didn't do much to control the degree or number of unknowns of the resulting polynomial. There doesn't even seem to be a published estimate of how big the original polynomial would be. Robinson and Matiyasevich, jointly and separately, produced a series of proofs reducing the needed number of unknowns, with Matiyasevich eventually reducing it to $9$ in an unpublished proof. Jones published this proof along with an analysis (some of which was performed by Wada) of various universal degree-unknown-number pairs, reporting the existence, for example, of a degree $4$ positively universal polynomial in $58$ unknowns and a $9$-unknown positively universal polynomial of degree $47216\cdot 5^{58} + 9728 \approx 1.638 \times 10^{45}$ \cite{Jones1982}.  Unfortunately Jones omits a fair amount of details of these constructions and only explicitly gives the polynomial $p(x_1,\dots,x_{28};n,m_1,m_2,m_3)$ of degree $2\cdot 5^{60}$ shown in Figure~\ref{fig:equation} along with the general technique for constructing universal polynomials with $9$ variables. 
Focusing on this polynomial is reasonable, as it has the virtue of being short enough to actually write on one page, but it's unclear how difficult it would be to fully replicate the other results reported by Jones. Using a standard technique of encoding intermediate calculations in extra variables \cite[Thm.~6.2]{Murty2019-pi}, it is possible to convert the polynomial in Figure~\ref{fig:equation} to an equivalent one of degree $4$ but with roughly $\log_2 5^{60} \approx 140$ variables.
%
%
%

%
%
%
%
%
%
%
%
%
%
%
%

\subsection{Polynomial engineering}

Now we have almost all of the ingredients we need to arrive at our main results. One last fact is the observation, originally due to Cantor, that there is a polynomial pairing function on $\Nb$.

\begin{fact}[Cantor]
  The polynomial $\frac{1}{2}(x+y)(x+y+1) + y$ defines a bijection between the set of pairs of non-negative integers and the set of non-negative integers. In particular, the polynomial $J_2(x,y) \coloneq (x+y)(x+y+1) + 2y$ is an injection from $\Nb^2$ into $\Nb$.
\end{fact}

For each $n \geq 2$, let $J_{n+1}(x_1,\dots,x_{n+1}) \coloneq J_2(J_n(x_1,\dots,x_n),x_{n+1})$. Note that each $J_n$ is a polynomial of degree $n$ with integer coefficients which moreover defines an injection from $\Nb^n$ to $\Nb$.

\begin{lem}\label{lem:balls}
  For any $n$, any sequence of rational pairs $(a_i,b_i)_{1 \leq i \leq n}$ with $a_i < b_i$ for each $i$, any real vector $\langle x_1,\dots,x_n \rangle \in \prod_{i=1}^n(a_i,b_i)$, and any $s \in \Rb$, there are natural numbers $k_1,\dots,k_{3+n}$ such that
  \[
    k_1^2k_2 - k_3 \sum_{i = 1}^n (k_1x_i -k_{3+i}+k_3)^2 > s
  \]
  but $k_1^2k_2 - k_3 \sum_{i = 1}^n (k_1z_i -k_{3+i}+k_3)^2 < 0$ for any $\langle z_1,\dots,z_n \rangle \notin \prod_{i=1}^n (a_i,b_i)$.
\end{lem}
\begin{proof}
  Let $r(\xbar,\kbar) = k_1^2k_2 - k_3 \sum_{i = 1}^n (k_1x_i -k_{3+i}+k_3)^2$. We clearly have that if $k_1$ and $k_3$ are positive, then the set of $\xbar$ for which $r(\xbar,\kbar) \geq 0$ is the closed Euclidean ball with center $\left\langle \frac{k_{4}-k_3}{k_1},\dots, \frac{k_{3+n}-k_3}{k_1}\right\rangle$ and radius $\sqrt{\frac{k_2}{k_3}}$. For any $\xbar \in \prod_{i=1}^n(a_i,b_i)$, we can find such a ball $B$ such that $B$ is a subset of $\prod_{i=1}^n(a_i,b_i)$ and $\xbar$ is in the interior of $B$. Moreover, we have that $r(\xbar,t\kbar) = t^3r(\xbar,\kbar)$. Therefore we can make the particular value of $r(\xbar,\kbar)$ arbitrarily large without changing the set $\{\zbar \in \Rb^n : r(\zbar,\kbar) \geq 0\}$. 
\end{proof}

\begin{lem}\label{lem:polynomial-engineering}
  Fix a positively universal polynomial $q(y_1,\dots,y_\nu;\ell,\mbar)$ of degree $\delta$. For any semicomputable set $U \subseteq \Rb^a\times \Nb$, there is a tuple $\mbar$ of natural numbers such that the polynomial
  \[
    p(x_1,\dots,x_a,n,y_1,\dots,y_\nu,\ell_1,\dots,\ell_{3+a},k_1,\dots,k_{3+a})
  \]
   of degree $\max\{3+\delta,7\}$ given by
  \begin{gather*}
k_1^2k_2\left(1-q(\ybar;\ell_{3+a},\mbar)-(\ell_1-J_2(n,k_1))^2-\sum_{i=1}^{2+a}(\ell_{i+1}-J_2(\ell_i,k_{i+1}))^2\right) 
    -k_3\sum_{i = 1}^a (k_1x_i-k_{3+i}+k_2)^2
  \end{gather*}
  satisfies that
  \[
    \sup_{\ybar\in\Nb^\nu,\ellbar \in \Nb^{3+a},\kbar \in \Nb^{3+a}}p(\xbar,n,\ybar,\ellbar,\kbar) =
    \begin{cases}
      1 & n = 0 \\
      +\infty & n > 0,~\langle \xbar,n - 1 \rangle \in U \\
      0 & n > 0,~\langle \xbar,n - 1 \rangle \notin U
    \end{cases}
  \]
  for any $\xbar \in \Rb^a$ and $n \in \Nb$.
\end{lem}
\begin{proof}
  Let $W_0$ be a computably enumerable set satisfying that $U = \bigcup_{t \in W_0}N(\Rb^a\times \Nb,t)$. Let $W_1$ be the set of $\ell \in \Nb$ satisfying that $\ell = J_{4+a}(n,\kbar)$ for some natural numbers $n$ and $\kbar$ such that one of the following holds: %
  \begin{enumerate}[(i)]
  \item\label{ce-one} $n = 0$, $k_1 = 1$, and $k_2 = k_3 = 0$,
  \item\label{ce-two} $n  > 0$ and $k_1 = k_3 = 0$, or
  \item\label{ce-three} $n,k_1,k_3 > 0$ and there is a $t \in W_0$ such that
    \[
      \{\langle x_1,\dots,x_a,n-1 \rangle : \xbar \in \Rb^a,~r(\xbar,\kbar)\geq 0\} \subseteq N(\Rb^a\times \Nb,t).
    \]
  \end{enumerate}
  Note that the subset inclusion in (\ref{ce-three}) is decidable as a function of $\langle n,\kbar \rangle$ and $t$, so we have that $W_1$ is a computably enumerable set. Let $\mbar$ be a tuple of natural numbers chosen so that for any $\ell \in \Nb$, $q(\ybar;\ell,\mbar)$ has a solution in the natural numbers if and only if $\ell \in W_1$.
  
  Throughout the rest of this proof, we'll say that \emph{the $\ell_i$'s are correct} to mean that $\ell_1 = J_2(n,k_1)$ and $\ell_{i+1} = J_2(\ell_i,k_{i+1})$ for all $i \leq 2+a$ (i.e., that the $\ell_i$'s have correctly coded the computation of $\ell_{3+a} = J_{4+a}(n,\kbar)$). We'll say that \emph{the $\ell_i$'s are not correct} to mean that $\ell_1 \neq J_2(n,k_1)$ or $\ell_{i+1} = J_2(\ell_i,k_{i+1})$ for some $i \leq 2+a$. Note that if the $\ell_i$'s are correct, then $(\ell_1-J_2(n,k_1))^2+\sum_{i=1}^{2+a}(\ell_{i+1}-J_2(\ell_i,k_{i+1}))^2 = 0$ and if the $\ell_i$'s are not correct, then $(\ell_1-J_2(n,k_1))^2+\sum_{i=1}^{2+a}(\ell_{i+1}-J_2(\ell_i,k_{i+1}))^2 > 0$.
  
  Let $r(\xbar,\kbar) = k_1^2k_2 - k_3 \sum_{i = 1}^n (k_1x_i -k_{3+i}+k_3)^2$ be the polynomial from \cref{lem:balls} again. Note that
  \[
    p(\xbar,n,\ybar,\ellbar,\kbar) = r(\xbar,\kbar) - k_1^2k_2\left(q(\ybar;\ell_{3+a},\mbar)+(\ell_1-J_2(n,k_1))^2+\sum_{i=1}^{2+a}(\ell_{i+1}-J_2(\ell_i,k_{i+1}))^2\right).
  \]
   We immediately have the following:
  \begin{itemize}
  \item If $q(\ybar;\ell_{3+a},\mbar) = 0$ and the $\ell_i$'s are correct, then $p(\xbar,n,\ybar,\ellbar,\kbar) = r(\xbar,\kbar)$.
  \item If $q(\ybar;\ell_{3+a},\mbar) > 0$ or the $\ell_i$'s are not correct, then the first term of $p(\xbar,n,\ybar,\ellbar,\kbar)$ is non-positive and so $p(\xbar,n,\ybar,\ellbar,\kbar) \leq 0$.
  \end{itemize}
  Note that these two cases are exhaustive by the positive universality of $q(\ybar;\ell_{3+a},\mbar)$ and the choice of $\mbar$.

  Let
  \(
    g(\xbar,n) = \sup_{\ybar\in\Nb^\nu,\ellbar \in \Nb^{3+a},\kbar \in \Nb^{3+a}}p(\xbar,n,\ybar,\ellbar,\kbar).
  \) We need to unpack the behavior of $g(\xbar,n)$ to verify that it satisfies the conclusion of the lemma.

  For $n= 0$, we have the following:
  \begin{itemize}
  \item If $k_1 = 1$, $k_2 = k_3 = 0$ and the $\ell_i$'s are correct, then $\ell_{3+a} \in W_1$ by (\ref{ce-one}) and so $p(\xbar,n,\ybar,\ellbar,\kbar) = 1$ for any choice of the other variables.
  \item If $k_1 \neq 1$, $k_2 \neq 0$, $k_3 \neq 0$, or the $\ell_i$'s are not correct, then $p(\xbar,n,\ybar,\ellbar,\kbar) \leq 0$ for any choice of the other variables.
  \end{itemize}
  Since these two cases are exhaustive, we have that $g(\xbar,1) = 1$ for any $\xbar$.

  For $n > 0$, we have the following:
  \begin{itemize}
  \item If $k_1 = k_3 = 0$ and the $\ell_i$'s are correct, then $\ell_{3+a} \in W_1$ by (\ref{ce-two}) and so $p(\xbar,n,\ybar,\ellbar,\kbar) = 0$ for any choice of the other variables.
  \item If the conditions in (\ref{ce-three}) are satisfied by $n$ and $\kbar$ and the $\ell_i$'s are correct, then $p(\xbar,n,\ybar,\ellbar,\kbar) = r(\xbar,\kbar)$.
  \item Otherwise if the conditions of the previous two bullet points fail, then $p(\xbar,n,\ybar,\ellbar,\kbar) \leq 0$.
  \end{itemize}
  The first bullet point implies that $g(\xbar,n) \geq 0$ for any $\xbar$ and $n > 0$. The definition of $W_1$ implies that for any $\xbar \notin U$, if $n > 0$ and $\kbar$ satisfy (\ref{ce-three}), then $p(\xbar,n,\ybar,\ellbar,\kbar) = r(\xbar,\kbar) \leq 0$. So for any $\xbar$ and $n > 0$ with $\langle \xbar, n-1 \rangle \notin U$, we have that $g(\xbar,n) = 0$. Finally, for any $\xbar$ and $n > 0$ with $\langle \xbar,n-1 \rangle \in U$, we have by \cref{lem:balls} that for any $s \in \Rb$, there are $\ybar$, $\ellbar$, and $\kbar$ such that $p(\xbar,n,\ybar,\ellbar,\kbar) = r(\xbar,\kbar) > s$. Therefore for any $\langle \xbar,n-1 \rangle \in U$, $g(\xbar,n) = + \infty$.
\end{proof}

The polynomial in \cref{lem:polynomial-engineering} was chosen to minimize degree. If instead our goal was to minimize the number of variables, we could have chosen
\[
  k_1^2k_2\left(1-q(\ybar;\ell,\mbar)-(\ell-J_{4+a}(n,\kbar))^2\right)-k_3\sum_{i = 1}^a (k_1x_i-k_{3+i}+k_2)^2
\]
giving a polynomial $p(x_1,\dots,x_a,n,y_1,\dots,y_\nu,\ell,k_1,\dots,k_{3+a})$ of degree $\max\{3+\delta,11+2a\}$.

It seems unlikely that the number of variables or degree of the polynomial in \cref{lem:polynomial-engineering} is optimal (even relative to the given $\nu$ and $\delta$). %
For instance, we pick up some complexity injecting $\langle n,\kbar \rangle$ into $\Nb$, which could probably be folded into the construction of the universal polynomial itself.

\begin{prop}\label{prop:main-G-delta-version}
  If there is a positively universal polynomial with $\nu$ unknowns of degree $\delta$, then for any (lightface) $\lPi^0_2$ set $X \subseteq \Rb^a$, there is a polynomial $p(x_1,\dots,\allowbreak x_a,n,k_1,\dots,\allowbreak k_{6+\nu+2a})$ of degree $\max\{3+\delta,7\}$ with integer coefficients such that
  \[
    f(\xbar) = \inf_{n \in \Nb}\sup_{\kbar \in \Nb^{6+\nu+2a}}p(\xbar,n,\kbar)
  \]
  is the indicator function of $X$.
\end{prop}
\begin{proof}
  Fix a $\lPi^0_2$ set $X \subseteq \Rb^a$. By definition, there is a $\lSigma^0_1$ set $U \subseteq \Rb^a\times \Nb$ such that $X = \{\xbar \in \Rb^a : (\forall n \in \Nb)\langle \xbar,n \rangle \in U\}$. By \cref{lem:polynomial-engineering}, we can find a polynomial $p(x_1,\dots,x_a,n,k_1,\dots,k_{6+\nu+2a})$ of the required degree such that $\sup_{\kbar \in \Nb^{6+\nu+2a}}p(\xbar,n,\kbar)$ is $1$ if $n = 1$, $+ \infty$ if $n > 1$ and $\langle \xbar,n - 1 \rangle \in U$, and $0$ if $n > 1$ and $\langle \xbar, n - 1 \rangle \notin U$. This implies that $\inf_{n \in \Nb} \sup_{\kbar \in \Nb^{6+\nu+2a}}p(\xbar,n,\kbar)$ is the indicator function of $X$.
\end{proof}

The key observation for the following corollary is that for indicator functions, the supremum operator is equivalent to a projection and the infimum operator is equivalent to a co-projection\footnote{The \emph{co-projection} of a set $A \subseteq X \times Y$ is the complement of the projection of the complement.} of the indicated set. This is directly related to the observation that the supremum is like a real-valued analog of existential quantification and the infimum is similarly analogous to universal quantification, which is a basic idea in real-valued logics. %

\begin{cor}\label{cor:main-prop}
  If there is a positively universal polynomial with $\nu$ unknowns of degree $\delta$, then for any (lightface) $\lSigma^1_{m}$ set $X \subseteq \Rb^a$, there is a polynomial $p(x_1,\dots,\allowbreak x_a,z_1,\dots,\allowbreak z_{m},n,k_1,\dots,k_{6+\nu+2a+2m})$ of degree $\max\{3+\delta,7\}$ with integer coefficients such that if $m$ is odd, then
  \[
    f(\xbar) = \sup_{z_1 \in \Rb}\inf_{z_2 \in \Rb}\cdots \inf_{z_{m-1} \in \Rb} \sup_{z_{m} \in \Rb} \inf_{n \in \Nb}\sup_{\kbar \in \Nb^{6+\nu+2a+2m}}p(\xbar,\zbar,n,\kbar)
  \]
  is the indicator function of $X$ and if $m$ is even, then 
  \[
    f(\xbar) = \sup_{z_1 \in \Rb}\inf_{z_2 \in \Rb}\cdots \sup_{z_{m-1} \in \Rb}\inf_{z_{m} \in \Rb} \sup_{n \in \Nb}\inf_{\kbar \in \Nb^{6+\nu+2a+2m}}p(\xbar,\zbar,n,\kbar)
  \]
  is the indicator function of $X$.
\end{cor}
\begin{proof}
  It follows from Lemmas~\ref{lem:replace-Baire} and \ref{prop:main-G-delta-version} by induction that for any odd $m$ and $\lSigma^1_m$ set $X \subseteq \Rb^a$ or any even $m > 0$ and $\lPi^1_m$ set $X \subseteq \Rb^a$, there is a polynomial $p$ of the required degree with integer coefficients such that $f(\xbar) = \sup_{z_1 \in \Rb}\inf_{z_2 \in \Rb}\cdots \inf_{z_{m-1} \in \Rb} \sup_{z_{m} \in \Rb} \inf_{n \in \Nb}\sup_{\kbar \in \Nb^{6+\nu+2a+2m}}p(\xbar,\zbar,n,\kbar)$ is the indicator function of $X$.

  The statement then follows for $\lSigma^1_m$ sets $X \subseteq \Rb^a$ for even $m$ by applying the preceding result to the complement $\Rb^a \setminus X$ and then taking $1-p(\xbar,\zbar,n,\kbar)$ to be the required polynomial.
\end{proof}

\begin{thm}\label{thm:polynomial-Vitali}
  There is a polynomial $p(x,y,z,n,k_1,\dots,k_{70})$ of degree $7$ with integer coefficients such that if $\VV=\LL$, then
  \[
    f(x) = \inf_{y \in \Rb}\sup_{z \in \Rb}\inf_{n \in \Nb}\sup_{\kbar \in \Nb^{70}}p(x,y,z,n,\kbar)
  \]
  is the indicator function of a Vitali subset of $[0,1]$, yet it is also relatively consistent with \ZFC\ that $f(x)$ is Lebesgue measurable. In particular, it is independent of \ZFC\ whether $f(x)$ is measurable.
\end{thm}
\begin{proof}
  By \cref{prop:Vitali-V-L} there is a $\lPi^1_2$ set $A \subseteq \Rb$ such that $\VV= \LL$ implies $A$ is a Vitali subset of $[0,1]$. Applying \cref{cor:main-prop} to the complement of $A$ (using the degree $4$ positively universal polynomial with $58$ unknowns in \cite{Jones1982}) to get a polynomial $q$ and taking $1-q$ gives the required polynomial $p(x,y,z,n,k_1,\dots,k_{70})$ of degree $7$.
\end{proof}

Of course, we can just as easily get that the indicator function of the well-ordering $<_L$ on $\Rb$ is similarly definable using a polynomial $p(x,y,z,w,n,k_1,\dots,k_{72})$ of degree $7$ with integer coefficients using \cref{fact:wo-defbl} and \cref{cor:main-prop}.
Another thing we might want to try when contemplating variants of \cref{thm:polynomial-Vitali} is minimizing the number of variables required rather than the degree. 
We can attain the same result with a polynomial $p(x,y,z,n,k_1,\dots,k_{14})$ of degree $47216\cdot 5^{58} + 9731$ using the suggested modification after \cref{lem:polynomial-engineering} and the $9$-variable universal polynomial in \cite{Jones1982}.

It should be noted that $\VV = \LL$ is consistent with the presence of many (although not all) large cardinals, including the existence of a proper class of inaccessible cardinals (and more, such as Mahlo and some Erd\H{o}s cardinals). This means that, on the one hand, the measurability of the function in \cref{thm:polynomial-Vitali} is independent of not just \ZFC, but of Tarski--Grothendieck set theory,\footnote{\emph{Tarski--Grothendieck set theory} is the theory $\ZFC + \text{``}$there is a proper class of inaccessible cardinals'', although the theory is not typically described this way specifically.} which is strong enough to formalize the vast majority of proofs that occur outside of the context of set theory itself. On the other hand, the potential non-measurability of $f(x)$ occurs in extremely weak theories too, as low as even fragments of second-order arithmetic. This is because the construction of $\LL$ is robust enough to work even in such weak contexts \cite[Sec.~VII.4]{Simpson2009}.

It is an open problem whether there is a universal Diophantine equation of degree $3$, but note that having such a polynomial would not decrease the degree needed for \cref{thm:polynomial-Vitali} (since \cref{lem:polynomial-engineering} gives a polynomial of degree $\max\{3+\delta,7\}$). This suggests an obvious question.

\begin{quest}\label{quest:degree}
  For which $\delta$ does \ZFC\ prove that all functions of the form
  \[
    f(x) = \inf_{y \in \Rb} \sup_{z \in \Rb} \inf_{n \in \Nb}\sup_{\kbar \in \Nb^m}p(x,y,z,n,\kbar)
  \]
   with $p$ a polynomial of degree $\delta$ are Lebesgue measurable?
\end{quest}

Just like with Diophantine universality, there are many possible variants of this question, such as asking about different values of $m$, but at the moment I can't even see how to resolve it for $\delta = 2$. None of the results in \cref{sec:obviously} are immediately helpful.

We can of course rewrite \cref{fact:Shelah-inaccess} in a form analogous to \cref{thm:polynomial-Vitali}.

\begin{thm}\label{thm:polynomial-inacc}
  There is a polynomial $p(x,y,z,w,t,n,k_1,\dots,k_{74})$ of degree $7$ with integer coefficients such that if the function
  \[
    g(x,y) = \sup_{z \in \Rb}\inf_{w \in \Rb}\sup_{t \in \Rb}\inf_{n \in \Nb}\sup_{\kbar \in \Nb^{74}}p(x,y,z,w,t,n,\kbar)
  \]
  satisfies that if $g(x,r)$ is Lebesgue measurable for every $r \in \Rb$, then $\aleph_1$ is an inaccessible cardinal in $\LL$. In particular, the statement ``$g(x,r)$ is measurable for every $r \in \Rb$'' implies the consistency of \ZFC.
\end{thm}
\begin{proof}
  By \cref{prop:sigma-universal}, there is a (lightface) $\lSigma^1_3$ set $X \subseteq \Rb^2$ that is universal for $\bSigma^1_3$ subsets of $\Rb$. We can apply \cref{cor:main-prop} to this $X$ together with the positively universal polynomial of degree $4$ with $58$ unknowns in \cite{Jones1982} to get a polynomial $p(x,y,z,w,t,n,k_1,\dots,k_{74})$ of degree $7$. The statement then follows by \cref{fact:Shelah-inaccess}.
\end{proof}

We should slow down for a second to comment on some of the remarkable properties of the function $g(x,y)$ in \cref{thm:polynomial-inacc}. For \emph{every} Borel (or more generally $\bSigma^1_3$) subset $X$ of $\Rb$, there is an $r \in \Rb$ such that $g(x,r)$ is the indicator function of $X$. The vast majority of named sets of reals in mathematics are Borel: The rational numbers, the irrational numbers, the algebraic numbers, the transcendental numbers, the set of normal numbers in any given base, the set of absolutely normal numbers, the set of Liouville numbers, every countable or co-countable set of real numbers, every closed or open set of real numbers, and many, many others can all be realized as sets of the form $\{x \in \Rb : g(x,r) = 1\}$ for some $r \in \Rb$. Moreover, the relevant $r$ can usually be computed explicitly, at least in principle. And there's also nothing particularly special about $\Rb$ here. By combining \cref{prop:sigma-universal} and \cref{cor:main-prop}, we can write a similarly universal function $g(\xbar,y)$ for any $\Rb^n$.

There are named non-Borel sets in certain Polish spaces that were not originally defined by logicians and were not constructed for the sake of being an example of a non-Borel or non-measurable set (such as the set of everywhere differentiable functions in the Banach space $C[0,1]$ of continuous functions on $[0,1]$ with the supremum norm), but I do not know of a single example in the reals themselves. Furthermore, most `naturally occurring' non-Borel sets are still $\bSigma^1_1$ or $\bPi^1_1$ (and therefore also $\bSigma^1_3$).

The following corollary is immediate but also worth highlighting.

\begin{cor}\label{cor:general-indep-statement}
  The statement
  \begin{itemize}
  \item[$ $] ``every function of the form
    \[
      f(x) = \qqq_{v \in \mathbf{X}} \qqq_{v \in \mathbf{X}} \qqq_{v \in \mathbf{X}}\cdots p(x,\dots)
    \]
    (where each $\qqq$ is $\sup$ or $\inf$, each $v$ is a variable, each $\mathbf{X}$ is $\Rb$ or $\Nb$, and $p$ is a polynomial with real coefficients) is Lebesgue measurable''
  \end{itemize}
  has the same consistency strength (over \ZFC) as the existence of an inaccessible cardinal. In particular, it implies the consistency of \ZFC\ and is therefore not provable in \ZFC.
\end{cor}
\begin{proof}
  One direction is immediate from \cref{thm:polynomial-inacc}. The other direction (i.e., showing that all such function can consistently be measurable assuming the existence of an inaccessible cardinal) follows from \cite[Thm.~2]{Solovay1970} and the fact that function $f(x)$ of the form in the given statement is $\bSigma^1_n$ for some $n$ (with $n$ at most one more than the number of infima and suprema in the definition of $f(x)$).
\end{proof}

Sets with even more specific structure can be constructed with analogously explicit functions under the assumption of $\VV = \LL$. We have put the majority of the proof of the following result in \cref{sec:explicit-BT}, but morally it is fairly similar to the proof of \cref{thm:polynomial-Vitali}.

\begin{thm}\label{thm:Polynomial-Banach-Tarski}
  There is a polynomial $p(m,x,y,z,w,t,n,k_1,\dots,k_{76})$ of degree $7$ with integer coefficients such that if $V=L$, then the functions $h(1,x,y,z),\dots,h(16,x,y,z)$ are the indicator functions of a paradoxical decomposition of $S^2$, where
  \[
    h(m,x,y,z) = \inf_{w \in \Rb} \sup_{t \in \Rb}\inf_{n \in \Nb}\sup_{\kbar \in \Nb^{76}}p(m,x,y,z,w,t,n,\kbar).
 \]
\end{thm}
\begin{proof}
  By \cref{cor:paradox-Pi}, there is a sequence $(A_i)_{i\leq 16}$ of $\lPi^1_2$ sets that, assuming $V=L$, are a paradoxical decomposition of $S^2$. By Facts~\ref{fact:Bool} and \ref{fact:function-sub}, this implies that the set $\{\langle \ell,x,y,z \rangle \in \Rb^4 : \ell \in \{1,\dots,16\},~\langle x,y,z \rangle \in A_i\}$ is $\lPi^1_2$, whereby the result follows from \cref{cor:main-prop}.
\end{proof}

Other highly choicy objects also admit lightface projective definitions under $\VV = \LL$, such as Hamel bases of $\Rb$ as a $\Qb$-vector space and non-principal ultrafilters on $\Nb$ encoded as subsets of the one-thirds Cantor set.

\subsection{What would it take to \emph{actually} actually write it down?}

Now comes the part of the paper where I must confess that I have not actually literally written down the polynomials proven to exist in Theorems \ref{thm:polynomial-Vitali}, \ref{thm:polynomial-inacc}, and \ref{thm:Polynomial-Banach-Tarski}. So what would one have to do to accomplish this?

The first issue, as already discussed, is that Jones's construction in \cite{Jones1982} of the $58$-unknown degree $4$ universal polynomial is not fully documented and would require a meaningful amount of work to reproduce carefully. Regardless, if our goal were to literally get a pen and a stack of paper and actually write one of these expressions down, it would be better to use the $28$-unknown degree $2\cdot 5^{60}$ polynomial shown in Figure~\ref{fig:equation}, as it was chosen to optimize the number of operations and therefore literal written length of the polynomial.\footnote{Or, if we were willing to go a little beyond polynomials, we could use the system of equations given in \cite[Thm. 1]{Jones1982}, which uses an exponential and some binomial coefficients but is less than half as long as the system that yields Figure~\ref{fig:equation} and only has $12$ unknowns.} Clearly it would be quite easy for us to write this in place of $q$ in the polynomial given in \cref{lem:polynomial-engineering}, write four suprema and infima in front of the resulting expression, and thereby `actually write down' something like the $f(x)$ in \cref{thm:polynomial-Vitali}, right?

Well, the next issue is that we would need to determine the parameters $m_1$, $m_2$, and $m_3$. Conceptually these code the Turing machine that enumerates the relevant computably enumerable set. Working out this Turing machine would require carefully writing out the coding of models of fragments of $\ZFC$ as elements of $\Nb^\Nb$, writing G\"odel's formula defining $L$ (which involves internal coding of syntax, much like his proof of the incompleteness theorems), using these to define $<_L$, and then unwinding the various results in effective descriptive set theory we used to transfer these constructions from $\Nb^\Nb$ to $\Rb$ and produce the relevant semicomputable set.  

A bigger factor in the complexity of the resulting expression might be the translation from a Turing machine to the coefficients $m_1$, $m_2$, and $m_3$. One might think that this would be fairly straightforward, but the issue is that work on optimizing the complexity of universal polynomials did not proceed by building more efficient encodings of Turing machines (or any of the other standard basic models of computation) into polynomials. What Jones actually did in \cite{Jones1982} is build a polynomial that is able to simulate \emph{another} polynomial of much higher degree,\footnote{In the case of the $9$-unknown polynomial, there is a third layer of simulation---polynomials within polynomials within polynomials. This is the reason that the number $58$ occurs in the degree of the $9$-unknown polynomial; its construction uses the $58$-unknown degree $4$ polynomial.} relying on the existing construction of a universal polynomial. This means that in order to go all the way, one would need to follow the steps of one of the direct proofs of \cref{fact:MRDP} (such as the one given in \cite{Murty2019-pi}). Since these proofs are written to be easier to understand conceptually, they aren't particularly concerned with controlling complexity. As far as I know there hasn't even been a published estimate of how big the resulting polynomial would be.

Despite how daunting these issues might seem, I think it's not unreasonable to think that we would be able to actually do all of this. I conjecture that it is possible to write down the full expression (with fully computed coefficients) of something like the $f(x)$ in \cref{thm:polynomial-Vitali} in the space of a relatively short book. Moreover, it is known that the task of simulating a Turing machine becomes easier if we allow ourselves larger (but still entirely mundane) sets of functions \cite{Mazzanti2002}. In particular, with these larger sets of functions it is possible to define projections for the Cantor pairing function directly, allowing us to reduce quantification over tuples of natural numbers to quantification over a single natural number. As such, I will also make the following conjecture.


\begin{conj}
  \begin{sloppypar}
    There is a function $g(x,y,z,n,k)$ that can be written in no more than 10 pages of standard \LaTeX\ using integers and the functions $a+b$, $a-b$, $a\cdot b$, $\left\lfloor \frac{a}{b}\right\rfloor$, $a^b$, $a!$, $\max\{a,b\}$, $\min\{a,b\}$, $\gcd(a,b)$, $\mathrm{lcm}(a,b)$, and $(a\ \mathrm{mod}\ b)$ such that it is independent of $\ZFC$ whether $f(x) = \inf_{y \in \Rb} \sup_{z \in \Rb} \inf_{n \in \Nb} \sup_{k \in \Nb} g(x,y,z,n,k)$ is Lebesgue measurable.
  \end{sloppypar}
\end{conj}







\section{Measurability by inspection}
\label{sec:obviously}

Definitions broadly like the one in \cref{thm:polynomial-Vitali} are fairy common in mathematics, yet it is also a relatively universal experience that, in practice, one does not need to worry about measurability of functions one has `actually written down.' %
To what extent can we explain this?

One suspect in \cref{thm:polynomial-Vitali} is the fact that polynomials are not in general uniformly continuous. Given any uniformly continuous function $g(x,y)$, $f(x) = \sup_y g(x,y)$ is also uniformly continuous and therefore measurable. This is a specific case of a general phenomenon, which is that often ostensibly projective sets or functions are actually Borel. (We'll see a more subtle example of this in \cref{sec:o-min}.) For instance, often `quantification' (e.g., suprema, infima, intersections, and unions) over an uncountable family can be replaced with quantification over some dense countable subfamily by some monotonicity or continuity condition.

\subsection{Measurability of analytic sets}

Roughly speaking,
the strongest\footnote{There are of course always stronger results, such as Solovay's result that `provably $\bDelta^1_2$' sets are Lebesgue measurable \cite[Ex.~14.4]{Kanamori2003}.} \ZFC-provable general theorem regarding measurability is \Lusin's result that analytic sets are always universally measurable.%

\begin{defn}
  A subset of a Polish space $X$ is \emph{universally measurable} if it is measurable with regards to the completion of any $\sigma$-finite Borel measure on $X$. A function $f : X \to \olR$ is \emph{universally measurable} if $f^{-1}((r,\infty])$ is universally measurable for every $r \in \olR$.
\end{defn}

Note that the set of universally measurable subsets of a Polish space is always a $\sigma$-algebra, since it is an intersection of $\sigma$-algebras.

\begin{fact}[\Lusin]\label{fact:analytic-univ-meas}
  Any $\bSigma^1_1$ (i.e., analytic) set is universally measurable. %
\end{fact}

Since analytic sets are precisely the images of Borel sets under Borel maps, this class is fairly broad. Furthermore, although complements of analytic sets are not in general analytic, the pointclass of analytic sets is closed under countable unions and intersections (\cref{fact:Bool}).

For the purpose of using \cref{fact:analytic-univ-meas} to show that certain classes of functions are universally measurable, we'll need the following definition and lemma. 

\begin{defn}\label{defn:semi-sigma}
  For any Polish space $X$, a function $f : X \to \olR$ is \emph{(lower) semi-$\bSigma^1_1$} if for every $r \in \olR$, $f^{-1}((r,\infty])$ is $\bSigma^1_1$.
\end{defn}

Unfortunately the term \emph{semianalytic} is both already in use and will appear later in this paper, so we have opted for the visually clunky but unambiguous term in \cref{defn:semi-sigma}. Strictly speaking, the notion of \emph{upper semi-$\bSigma^1_1$} function clearly also makes sense, but to avoid an overabundance of words we will only use the term `semi-$\bSigma^1_1$' and this will always refer to the notion in \cref{defn:semi-sigma}.

\begin{lem}$ $\label{lem:semi-sigma}
  \begin{enumerate}
  \item\label{Borel-char} A function $f$ is Borel if and only if $f$ and $-f$ are both semi-$\bSigma^1_1$. In particular, Borel functions are semi-$\bSigma^1_1$.
  \item\label{univ-meas} Any semi-$\bSigma^1_1$ function is universally measurable.
  \item\label{closed-suplevel} $f$ is semi-$\bSigma^1_1$ if and only if for any $r \in \olR$, $f^{-1}([r,\infty])$ is $\bSigma^1_1$.
  \item\label{countable-quant} If $f: X \times Y \to \olR$ is semi-$\bSigma^1_1$, then $g(x) = \sup_{a \in A}f(x,a)$ and $h(x) = \inf_{a \in A}f(x,a)$ are semi-$\bSigma^1_1$ for any countable $A \subseteq Y$.
  \item\label{uncountable-quant} If $f: X \times Y \to \olR$ is semi-$\bSigma^1_1$, then $\ell(x) = \sup_{y \in Y}f(x,y)$ is semi-$\bSigma^1_1$.
  \end{enumerate}
\end{lem}
\begin{proof}
  \ref{Borel-char} follows from \ref{closed-suplevel} and \cref{fact:Suslin}. \ref{univ-meas} follows immediately from \cref{fact:analytic-univ-meas}. \ref{closed-suplevel} follows from $f^{-1}([r,\infty]) = \bigcap_{s < r,~s \in \Qb}f^{-1}((r,\infty])$, $f^{-1}((r,\infty]) = \bigcup_{s > r,~s \in \Qb}f^{-1}([r,\infty])$ and \cref{fact:Bool}.

  For \ref{countable-quant}, first note that
  \(
    g^{-1}((r,\infty]) = \bigcup_{a \in A}\{ x \in X : \langle x,a \rangle \in f^{-1}((r,\infty])).
  \)
  Since $X \times \{a\}$ is closed in $X \times Y$ for each $a \in A$, the terms in the union are $\bSigma^1_1$, so we have that $g$ is semi-$\bSigma^1_1$. For $h$, we similarly have that $h^{-1}([r,\infty]) = \bigcap_{a \in A}\{x \in X : \langle x,a \rangle \in g^{-1}([r,\infty])\}$, implying that each $h^{-1}([r,\infty])$ is $\bSigma^1_1$ and therefore $h$ is semi-$\bSigma^1_1$ by \ref{closed-suplevel}.

  Finally, for \ref{uncountable-quant}, note that $\ell^{-1}((r,\infty]) = \pi(f^{-1}((r,\infty])))$, where $\pi : X \times Y \to X$ is the projection. Since projections of $\bSigma^1_1$ sets are $\bSigma^1_1$, we have that $\ell$ is semi-$\bSigma^1_1$.
\end{proof}

\subsection{An extra quantifier from \texorpdfstring{$\sigma$}{σ}-compactness}
\label{sec:extra-quant-sigma-compact}

Another suspect in \cref{thm:polynomial-Vitali} is alternation of quantifiers (i.e., alternation of suprema and infima). We can show that $4$ alternating blocks of quantifiers are necessary for something like it to work.

Recall that a topological space is \emph{$\sigma$-compact} if it is a countable union of compact sets. Also recall that a function $g : X \to \olR$ is \emph{lower semi-continuous} if the open superlevel sets $g^{-1}((r,\infty])$ are open for each $r \in \olR$. Note that continuous functions in particular are lower semi-continuous.

\begin{prop}\label{prop:three-quant-insufficient}
  Fix a $\sigma$-compact Polish space $Y$ and an arbitrary $\sigma$-compact topological space $Z$. Fix an arbitrary set $W$ (which we will think of as a discrete topological space). Let $g : Y \times Z \times W \to \olR$ be a lower semi-continuous function. Then the function $f : Y \to \olR$ defined by
  \[
    f(y) = \inf_{z \in Z}\sup_{w \in W}g(y,z,w)
  \]
  is Borel. 
\end{prop}
\begin{proof}
  The function $f_0(y,z) = \sup_{w \in W}g(y,z,w)$ is lower semi-continuous, since it is the supremum of a family of lower semi-continuous functions. Therefore the closed sublevel sets $f_0^{-1}([-\infty,r]) \coloneq \{(y,z) : f_0(y,z) \leq r\}$ are closed. Since $Y$ and $Z$ are $\sigma$-compact, their product is as well and so each of the sets $f_0^{-1}([-\infty,r])$ is $\sigma$-compact.

  Now $f(y) = \inf_{z \in Z}f_0(y,z)$. We have that $f^{-1}([-\infty,r)) \coloneq \{y : f(y) < r\}$ is equal to $\bigcup_{s < r}\pi_Z(f_0^{-1}([-\infty,s]))$ (where $\pi_Z : Y \times Z \to Z$ is the canonical projection map). By monotonicity, we also have that $f^{-1}([-\infty,r)) = \bigcup_{s < r,~s \in \Qb}\pi_Z(f_0^{-1}([-\infty,s]))$. For each $s$, we have that since $f_0^{-1}([-\infty,s])$ is $\sigma$-compact, $\pi_Z(g_0^{-1}([-\infty,s]))$ is as well and is hence also $F_\sigma$ (since $Y$ is Hausdorff). Therefore each $f^{-1}([-\infty,r))$ is also $F_\sigma$, since it is a countable union of $F_\sigma$ sets.
\end{proof}

\begin{cor}\label{cor:three-quant-insufficient-1}
  If $X$, $Y$, $Z$, and $W$ are Polish spaces (with $Y$ and $Z$ $\sigma$-compact) and $g : X \times Y \times Z \times W \to \olR$ is lower semi-continuous, then
  \[
    f(x) = \sup_{y \in Y}\inf_{z \in Z}\sup_{w \in W}g(x,y,z,w)
  \]
   is semi-$\bSigma^1_1$ and therefore universally measurable.
\end{cor}

One thing to note is that \cref{prop:three-quant-insufficient} does need $\sigma$-compactness. Without $\sigma$-compactness, the best we can guarantee is that if $X$ and $Y$ are Polish spaces and $Z$ is an arbitrary set, then for any continuous $g : X \times Y\times Z \to \olR$, $f(x) = \inf_{y \in Y} \sup_{z \in Z}g(x,y,z)$ is semi-$\bSigma^1_1$.

\subsection{Logical tameness and o-minimality}
\label{sec:o-min}

Another culprit in \cref{thm:polynomial-Vitali} is the \emph{interaction} between quantifying over $\Rb$ and quantifying over $\Nb$. It is of course a standard fact that suprema and infima of countable families of measurable functions are measurable (since measurable sets form a $\sigma$-algebra), but it turns out that if we start from `logically tame' functions, solely quantifying over $\Rb$ can be tame as well.

The earliest named example of this phenomenon is probably the Tarski--Seidenberg theorem, which says that projections of semialgebraic sets are semialgebraic, where a set $X \subseteq \Rb^n$ is \emph{semialgebraic} if it is a finite Boolean combination of polynomial equalities and inequalities (i.e., sets of the form $\{\xbar \in \Rb^n : p(\xbar) = 0\}$ and $\{\xbar \in \Rb^n : p(\xbar) > 0\}$ for polynomial $p$). Semialgebraic sets are locally closed and hence always Borel. 

Since complements of semialgebraic sets are semialgebraic, it follows immediately from the Tarski--Seidenberg theorem that every set $X \subseteq \Rb^n$ that is first-order definable in the reals as an (ordered) field is semialgebraic. In the context of this paper, this has the following consequence.

\begin{fact}[Tarski--Seidenberg]\label{fact:Tarsk-Seidenberg}
  Every function of the form
  \[
    f(x) = \qqq_{v \in \Rb} \qqq_{v \in \Rb} \qqq_{v \in \Rb}\cdots p(x,\dots)
  \]
  (where each $\qqq$ is $\sup$ or $\inf$, each $v$ is a variable, and $p$ is a polynomial with real coefficients) is Borel and therefore universally measurable.
\end{fact}
\begin{proof}
  The $\olR$-valued function $f$ is first-order definable\footnote{We say that a function $f : \Rb^n \to \olR$ is \emph{definable} in some structure on $\Rb$ if the sets $\{\langle \xbar ,y  \rangle \in \Rb^{n+1} : f(\xbar) = y\}$, $\{\xbar \in \Rb^n : f(\xbar) = \infty\}$, and $\{\xbar \in \Rb^n : f(\xbar) = -\infty\}$ are all definable in that structure.} with parameters in $(\Rb,+,\cdot)$. Hence each set of the form $f^{-1}((r,\infty])$ is semialgebraic and thereby Borel.
\end{proof}

This can be extended beyond polynomials to anything first-order definable in $(\Rb,+,\cdot)$, which includes things like $\sqrt{x}$ and rational function (once we fix some reasonable convention for handling normally undefined values). For instance, the function
\[
  f(x) =
  \begin{cases}
    \sqrt{x} & x \geq 0 \\
    0 & x < 0
  \end{cases}
\]
is definable in $(\Rb,+,\cdot)$ by the formula $\varphi(x,y) \equiv (x < 0 \wedge y = 0) \vee (x \geq 0 \wedge y \geq 0 \wedge y \cdot y = x)$.\footnote{Of course the relation $x \leq y$ is definable by the formula $\psi(x,y) \equiv \exists z(x + z \cdot z = y)$.}

In the 80s and into the 90s, seminal model-theoretic work of van den Dries, Knight, Pillay, and Steinhorn isolated a general class of structures, the \emph{o-minimal structures}, that admit this kind of automatic topological tameness for definable objects. This extended earlier work of \L ojasiewicz and Hironaka on \emph{semianalytic} and \emph{subanalytic} sets, which are certain sets defined in terms of equalities and inequalities of real-analytic functions and which exhibit topological tameness analogous to that of semialgebraic sets (see \cite[Ch.~5]{Acquistapace2022}).

\begin{defn}
  A first-order structure $\Rc$ extending $(\Rb,<)$ is \emph{o-minimal} if any definable (with parameters) subset of $\Rb$ is a finite union of intervals (of possibly $0$ or infinite length).
\end{defn}

Note that since semialgebraic subsets of $\Rb$ are always finite unions of intervals, we have that $(\Rb, +, \cdot)$ is o-minimal. This property ends up having profound structural consequences---for instance it implies that every definable set in any number of dimensions has finitely many connected components---and it has found many applications outside of model theory \cite{Binyamini2023,TQFT2024}. This is sometimes regarded as an answer to Grothendieck's challenge to `investigate classes of sets with the tame topological properties of semialgebraic sets' issued in his \emph{Esquisse d’un Programme}. 

\begin{fact}[{Knight, Pillay, Steinhorn \cite{Knight1986}, see also \cite[Ch.~3]{Dries1998}}]
  If $f : \Rb^k \to \olR$ is definable in an o-minimal structure $\Rc$, then for every (open, closed, or half-open) interval $I \subseteq \olR$, $f^{-1}(I)$ is a finite union of connected submanifolds of $\Rb^k$.
\end{fact}

If $f(\xbar,y)$ is definable in an o-minimal structure $\Rc$, then the functions $\inf_{y \in \Rb} f(\xbar,y)$ and $\sup_{y \in \Rb}f(\xbar,y)$ are as well. This means that we are able to generalize \cref{fact:Tarsk-Seidenberg} to any o-minimal structure on $\Rb$.

\begin{cor}\label{cor:o-min-meas}
  For any function $g(x,\dots)$ definable in an o-minimal structure $\Rc$, every function of the form
  \[
    f(x) =  \qqq_{v \in \Rb}\qqq_{v \in \Rb}\qqq_{v \in \Rb} \dots g(x,\dots)
  \]
  (where each $\qqq$ is $\sup$ or $\inf$ and each $v$ is a variable) is Borel and therefore universally measurable.
\end{cor}

By Wilkie's theorem \cite{Wilkie1996}, the structure $(\Rb,+, \cdot,e^x)$ is o-minimal, so \cref{cor:o-min-meas} applies to all functions built out of the basic arithmetic operations, $e^x$, $\log x$, and $|x|$, such as $\frac{\log |x|+e^{(y^2-z)/e^x}}{e^{y}+z^3}$.

Not all analytic functions are definable in o-minimal structures. In fact, even $\sin(x)$ is `wild' from this perspective.

\begin{fact}\label{fact:sine-sin}
  A function $f: \Rb^n \to \Rb$ is first-order definable (with parameters) in $(\Rb,+,\cdot,\sin(x))$ if and only if it is $\bSigma^1_n$ for some $n$. 
\end{fact}

\cref{fact:sine-sin} seems to be folklore, but a lot of the machinery needed to prove it is similar to the techniques used in this paper. The key is that $\sin(x)$ allows us to define $\mathbb{Z}$ and thereby define $\mathbb{N}$. This means that \cref{fact:sine-sin} holds for other expansions of $(\Rb,+,\cdot)$ by other seemingly tame functions such as $(\Rb,+,\cdot,\lfloor x \rfloor)$ (where $\lfloor x \rfloor$ is the floor function).

However, there is still some sense in which these familiar functions are tame in terms of quantification and measurability. This was originally studied in the context of semianalytic and subanalytic sets, but for our purposes it will be easier (and slightly more general) to phrase things in terms of o-minimality.

\begin{defn}
  A function $f : \Rb^n \to \Rb$ is \emph{locally o-minimal} if there is an o-minimal expansion of the reals in which the restriction of $f$ to $[-k,k]^n$ is definable for each $k$ (although not necessarily uniformly).
\end{defn}

By the main results of \cite{vandenDries1994}, every entire real-analytic function is locally o-minimal (including in particular $\sin(x)$). Furthermore, so are many commonly used piecewise continuous functions, such as the floor and ceiling functions and the square, sawtooth, and triangle wave functions.

This regularity assumption only really buys us one additional quantifier, which is non-trivial but only just.

\begin{prop}\label{prop:local-o-min}
  If $h : \Rb^{n + m} \to \Rb$ is locally o-minimal, then $f(\xbar) = \sup_{\ybar \in \Rb^m}h(\xbar,\ybar)$ and $g(\xbar) = \inf_{\ybar \in \Rb^m}h(\xbar,\ybar)$ are Borel and therefore universally measurable.
\end{prop}
\begin{proof}
  We have that $f(\xbar) = \sup_{k \in \Nb} \sup_{\ybar \in [-k,k]^m}h(\xbar,\ybar)$. By local o-minimality, $\sup_{\ybar \in [-k,k]^m}h(\xbar,\ybar)$ is Borel for each $k \in \Nb$, whereby $f(\xbar)$ is Borel. The infimum case follows from the supremum case by considering $-\sup_{\ybar \in \Rb^m}-h(\xbar,\ybar)$.
\end{proof}

Note though that in the case of real-analytic functions, \cref{prop:local-o-min} is not an improvement over \cref{prop:three-quant-insufficient}.

\subsection{Measurability from large cardinals}
\label{sec:meas-large}

It has been an ongoing theme of set-theoretic research that there is an intimate relationship between measurability of definable subsets of $\Rb$ and large cardinals. We've already seen a bit of this in \cref{fact:Shelah-inaccess}. We've also seen (in \cref{fact:Krivine-proj-meas}, for instance) that it can \emph{consistently} be the case that all reasonably definable functions on the reals are measurable, but on a psychological level it would be more comfortable if, for instance, it just were the case that no `explicitly definable' subset of $\Rb^3$ can be a paradoxical decomposition of $S^2$. The impossibility of such things follows from certain beyond-\ZFC\ set-theoretic principles that are commonly considered, such as various large cardinal axioms and \emph{projective determinacy} (a restricted form of the more famous \emph{axiom of determinacy} (\AD) which, unlike \AD, is consistent with the axiom of choice). We don't have enough space in this paper to discuss projective determinacy other than to mention that, like some large cardinal axioms, it has found applications outside of set theory and descriptive set theory, such as in recent work by Avil\'es, Rosendal, Taylor, and Tradacete resolving certain problems in Banach space theory under the assumption of projective determinacy \cite{ARTT2024}.

\begin{defn}
  An uncountable cardinal $\kappa$ is \emph{measurable} if there is a $\kappa$-additive $\{0,1\}$-valued measure $\mu$ on the $\sigma$-algebra of all subsets of $\kappa$ such that $\mu(\kappa)=1$ (where \emph{$\kappa$-additivity} in this context means that for any family $(A_i)_{i \in I}$ of subsets of $\kappa$ with $|I| < \kappa$, $\mu\left( \bigcup_{i \in I}A_i \right) = 1$ if and only if there is an $i \in I$ such that $\mu(A_i) = 1$).
\end{defn}

Measurable cardinals are inaccessible, but this is a bit like saying that the core of the sun is warm. 
To give a little perspective, the number of distinct inaccessible cardinals less than a measurable cardinal $\kappa$ is $\kappa$. To give a little bit more perspective, the number of cardinals less than a measurable $\kappa$ with this property is also $\kappa$. Moreover, the number of cardinals $\lambda$ less than $\kappa$ satisfying that ``the number of cardinals $\gamma$ less than $\lambda$ satisfying that `the number of inaccessible cardinals less than $\gamma$ is $\gamma$' is $\lambda$'' is $\kappa$. Furthermore, none of the obvious (even transfinite) iterations of this idea come close to actually characterizing measurability. If $\kappa$ is measurable, then there is a $\kappa$-additive measure $\mu$ on $2^\kappa$ such that
\[
  \mu(\{\alpha < \kappa : \alpha~\text{is an inaccessible cardinal}\}) = 1, 
\]
so there is some sense in which you can `randomly' pick an ordinal less than $\kappa$ and almost surely find an inaccessible cardinal. 


Measurable cardinals were defined by Ulam in the investigation of the question of which sets can admit countably additive measures on their full power sets (a generalization of the measure problem of Lebesgue),\footnote{In particular, the existence of a measurable cardinal follows from the assumption that there is a non-trivial countably additive $\{0,1\}$-valued measure on the full power set of some infinite set.} but in some sense the fact that measurable cardinals have implications for measurability in $\Rb$ in particular is a historical coincidence.

\begin{fact}[{Solovay \cite[Thm.~14.3]{Kanamori2003}}]\label{fact:meas-meas}
  If there is a measurable cardinal, then every $\bSigma^1_2$ set of reals is universally measurable.
\end{fact}

\cref{fact:meas-meas} is known to not be sharp (as in weaker large cardinal hypotheses, such as the existence of a Ramsey cardinal, suffice), but measurable cardinals show up in some surprising places in mathematics. The existence of a measurable cardinal is equivalent to the existence of a discrete topological space that is not realcompact (i.e., isomorphic to a closed subspace of $\Rb^I$ for any $I$) and to the existence of an exact functor $F : \Set \to \Set$ that is not naturally isomorphic to the identity (as discovered independently by Trnkov\'a \cite{Trnkova1971} and then later Blass \cite{blass1976exact}). %
It was shown by Prze\'zdziecki in \cite{AdamPrzezdziecki2006} that the \emph{non-}existence of a measurable cardinal is equivalent to the statement that every group is the fundamental group of a compact Hausdorff space. Directly concatenating this with \cref{fact:meas-meas} gives the somewhat absurd result that the existence of a group that is not the fundamental group of any compact Hausdorff space (non-vacuously) implies that the projection of the complement of the projection of any Borel subset of $\Rb^3$ is Lebesgue measurable.
In the context of this paper more specifically, the functions in Theorems~\ref{thm:polynomial-Vitali} and \ref{thm:Polynomial-Banach-Tarski} are measurable if there is a measurable cardinal. A corollary of this fact (which is much easier to prove directly and was originally shown by Scott in \cite{SCOTT1961}) is that there cannot be a measurable cardinal in $L$; it is `too thin' to accommodate such an object.\footnote{Note though that if there is a measurable cardinal $\kappa$, then the literal ordinal $\kappa$ will still exist in $L$ and will be a (fairly large inaccessible) cardinal there. It just won't be a measurable cardinal in the sense of having a countably additive $\{0,1\}$-valued measure on its power set.} There is however a modified construction, $L[U]$, which has many similar properties to $L$ and can accommodate a measurable cardinal. Silver showed in \cite{Silver1971Meas} that there is a $\lDelta^1_3$ well-ordering of $\Rb$ in $L[U]$,\footnote{See also \cite[Thm.~20.18]{Kanamori2003} for a modern proof of this in terms of \emph{iterable pre-mice}, which are an important tool in studying the fine structure of more sophisticated inner models like $L[U]$.} so we get analogs of Theorems~\ref{thm:polynomial-Vitali} and \ref{thm:Polynomial-Banach-Tarski} for the theory $\ZFC + \text{``}$there exists a measurable cardinal'' at the cost of adding one additional supremum or infimum at the beginning of the expression defining the relevant indicator functions.

There are stronger large cardinal assumptions which push the automatic measurability of simply defined functions up the projective hierarchy. The most precise version of these involve the concept of a \emph{Woodin cardinal},\footnote{Woodin cardinals are not usually defined this way, but Woodinness of $\kappa$ is equivalent to $H_\kappa$ satisfying the \emph{weak Vop\v{e}nka's principle} \cite[Def.~6.21]{Adamek1994}, which says that $\Ord^{\mathrm{op}}$ cannot be fully embedded into any locally presentable category \cite{Wilson2021}. In other words, $\kappa$ is a Woodin cardinal if and only if $\kappa^{\mathrm{op}}$ (as a posetal category) cannot be fully embedded into any locally $\kappa$-presentable category of size $\kappa$.} but the required assumption is much weaker than \emph{Vop\v{e}nka's principle}, which is more well known and has already found applications outside of logic. As discussed in \cite[Ch.~6]{Adamek1994}, Vop\v{e}nka's principle is equivalent to many nice statements in the theory of locally presentable categories. For instance, it is equivalent to the statement that a category $C$ is locally presentable if and only if it is co-complete and has a small full subcategory $D$ such that every object of $C$ is a colimit of a diagram in $D$.

\begin{fact}[{Shelah, Woodin \cite{Shelah1990}}] $ $\label{fact:Woodin-meas}
  \begin{enumerate}
  \item For any $n \in \Nb$, if there is a measurable cardinal that is larger than $n$ Woodin cardinals, then every $\bSigma^1_{n+2}$ set of reals is universally measurable.
  \item\label{second-meas} If there is a measurable cardinal that is larger than infinitely many Woodin cardinals, then every set of reals in $L(\Rb)$ is universally measurable. In particular, this follows from Vop\v{e}nka's principle. %
  \end{enumerate}
\end{fact}
\begin{proof}
  By \cite[Thm.~20.24, Lem.~20.25]{Jech2003}, Vop\v{e}nka's principle implies the existence of a supercompact cardinal, which is known to imply the assumption of \ref{second-meas} (as discussed in \cite{Shelah1990}). The rest follows from \cite[Thm.~32.9]{Kanamori2003}.
\end{proof}

  \begin{figure}
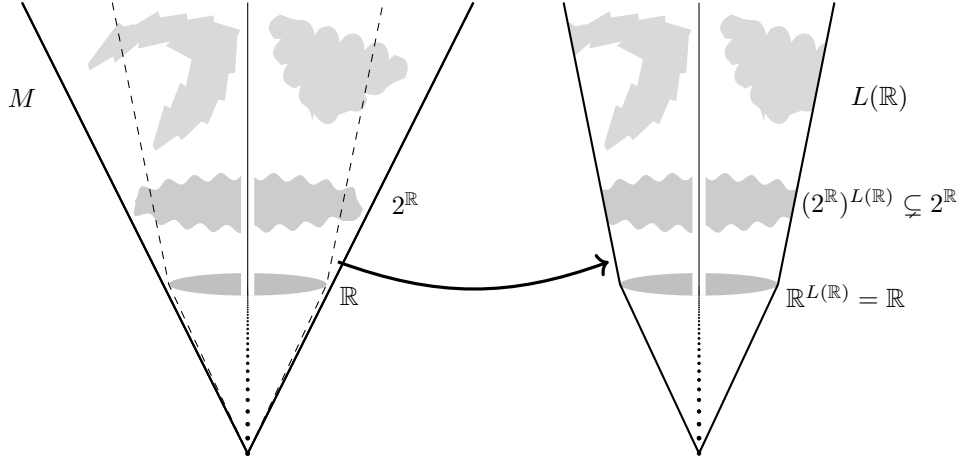

    \centering
    \LRFigure
    \caption{$L(\Rb)$ is the smallest inner model containing all of $\Rb$.}%
  \label{fig:LR}
\end{figure}
  
$L(\Rb)$ is the smallest inner model containing all of the real numbers (see Figure~\ref{fig:LR}). Unlike $L$ though, it doesn't always satisfy the axiom of choice (otherwise it would always have non-measurable sets). I would argue that $L(\Rb)$ encompasses (and probably over-approximates) most mathematicians' intuitive sense of what it means to `define something without the axiom of choice' or in other words to `actually write something down.' Any set of reals or real-valued function that one can prove the existence of without choice will live inside $L(\Rb)$, since it is a model of \ZF. So we have that if Vop\v{e}nka's principle holds, no analogs of Theorems~\ref{thm:polynomial-Vitali} and \ref{thm:Polynomial-Banach-Tarski} are possible. There can be no `explicit' description of a Vitali set or a Banach--Tarski decomposition of the sphere in the presence of sufficiently strong large cardinals.

\subsection{Table of (non-)measurability results}

\cref{tab:meas-tab} summarizes some of the various measurability results we've seen in this section and compares them to the non-measurability results from \cref{sec:construction} (and a couple we'll discuss in a moment). In this table, $\qqq$ represents an instance of either $\inf$ or $\sup$. The asterisks are Kleene stars, so in particular $\sup^\ast_X$ represents an arbitrarily long finite string of suprema over $X$ and $\qqq^\ast_X$ represents an arbitrarily long finite string of (possibly alternating) suprema and infima. The expression $[\sup_\Rb,\qqq_\Nb]^\ast$ represents an arbitrarily long finite string of $\sup_\Rb$'s, $\sup_\Nb$'s, and $\inf_\Nb$'s (again, possibly alternating).  Of course, all of these statements entail the analogous inverted statement, with suprema and infima (and lower and upper) switched, as all of the definability classes mentioned are closed under negation (except lower semi-continuity). Finally, $\mathrm{Con}(\ZFC)$ is the statement that $\ZFC$ is consistent.

Of course the table is not comprehensive. For instance, by Facts~\ref{fact:Bool} and \ref{fact:meas-meas}, the presence of a measurable cardinal extends the seventh entry of the table to $[\inf_{\Rb},\qqq_\Nb]^\ast\allowbreak[\sup_{\Rb},\qqq_{\Nb}]^\ast\allowbreak\inf_{\Rb,\Nb}^\ast\sup_{\Rb,\Nb}^\ast g$ for continuous $g$ (with analogous extensions for the eighth and ninth entries and in the presence of stronger large cardinal assumptions). There are also more independence results as seen in the third and fourth entries. %
For example, it is possible (although surprisingly tedious\footnote{\textbf{Exercise 1.}  Show that for any open set $U \subseteq \Rb^{n}$, there is an entire real-analytic function $g(\xbar,y)$ such that \(\mathbf{1}_U(\xbar) =  \sup_{y \in \Rb}g(\xbar,y) \) for all $\xbar \in \Rb^n$, where $\mathbf{1}_U(\xbar)$ is the indicator function of $U$.  For extra credit, argue informally that $g(\xbar,y)$ can be computable if $U$ is semicomputable.

 \textbf{Exercise 2.} Show that for any open set $U_0 \subseteq \Rb^k \times \Nb$, there is an open set $U \subseteq \Rb^{k+1}$ such that for any $\xbar \in \Rb^k$, $\langle \xbar,n \rangle \in U_0$ for all $n \in \Nb$ if and only if $\langle \xbar,y \rangle \in U$ for all $y \in \Rb$.

\textbf{Exercise 3.} Use Exercises~1 and 2 to show that for any $G_\delta$ set $A \subseteq \Rb^\ell$, there is a real-analytic function $g(\xbar,y,z)$ such that
  \(
    \mathbf{1}_A(\xbar) = \inf_{y \in \Rb}\sup_{z \in \Rb}g(\xbar,y,z)
  \)
  for all $\xbar \in \Rb^\ell$. Use this and earlier results in this paper to justify the third entry of \cref{tab:meas-tab}.}) to produce a computable total real-analytic function $g(x,y,z,w,t)$ such that it is independent of \ZFC\ whether $f(x) = \inf_{y \in \Rb}\sup_{z \in \Rb}\inf_{w \in \Rb}\sup_{t \in \Rb}g(x,y,z,w,t)$ is Lebesgue measurable, which shows that local o-minimality really doesn't allow us to improve \cref{cor:three-quant-insufficient-1} in general.

  \begin{table}[]
    \begin{adjustbox}{max width=\textwidth}
    \begin{tabular}{rll}
      Function definition & Definability of $g$ & Measurability \\ \hline \hline
      $\inf_{\Rb}\sup_{\Rb}\inf_{\Nb}\sup_{\Nb}^\ast g$                  & Polynomial over $\Zb$        & Independent (\ref{thm:polynomial-Vitali})      \\    %
      $\sup_{\Rb}\inf_{\Rb}\sup_{\Rb}\inf_{\Nb}\sup_{\Nb}^\ast g$                  & Polynomial over $\Rb$        & Implies $\mathrm{Con}(\ZFC)$ (\ref{thm:polynomial-inacc})      \\ %
      $\inf_{\Rb}\sup_{\Rb}\inf_{\Rb}\sup_{\Rb} g$                  & Real-analytic        &   Independent \\ %
      $\inf_{\Rb}\sup_{\Rb}\inf_{\Rb}^\ast\sup_{\Rb}^\ast\inf_{\Rb} g$                  & Trig.~polynomial    & Independent  \\ 
      $\inf^\ast_{\Rb}\sup^\ast_{\Rb}\inf_{\Rb}^\ast\sup_{\Rb}^\ast g$ & Trig.~polynomial & ????? (\ref{quest:trig}) \\ \hline

      $[\sup_{\Rb},\qqq_{\Nb}]^\ast g$                  & Borel        & Measurable (\ref{lem:semi-sigma})      \\ %
      $[\sup_{\Rb},\qqq_{\Nb}]^\ast\inf_{\Rb,\Nb}^\ast\sup_{\Rb,\Nb}^\ast g$                  & Continuous        & Measurable (\ref{lem:semi-sigma}, \ref{cor:three-quant-insufficient-1})      \\ %
      $[\sup_{\Rb},\qqq_{\Nb}]^\ast\qqq_{\Rb}^\ast g$                  &  O-minimal: Rational,         & Measurable (\ref{lem:semi-sigma}, \ref{cor:o-min-meas})      \\ %
                         &  Exponential, Step, etc.        &       \\  %
                          &  (e.g., $e^{x^2/|y+\mathbf{1}_{[3,4]}(\log z)|}$)        &       \\ %
     $[\sup_{\Rb},\qqq_{\Nb}]^\ast\inf_{\Rb}^\ast g$  & Locally o-minimal: & Measurable (\ref{lem:semi-sigma}, \ref{prop:local-o-min}) \\ %
        & Real-analytic, Floor, etc. &  \\  
                          & (e.g., $\lceil x^2\cos(e^{y\lfloor z \rfloor})\rceil$) & \\ \hline   %
$[\inf_{\Rb},\qqq_{\Nb}]^\ast[\sup_{\Rb},\qqq_{\Nb}]^\ast g$ & Borel &  Measurable cardinal (\ref{fact:meas-meas}) \\
      $\qqq_{\Rb,\Nb}^\ast g$ & $L(\Rb)$ & Vop\v{e}nka's principle (\ref{fact:Woodin-meas})
    \end{tabular}
  \end{adjustbox}
  \caption{Measurability of various classes of definable functions.}
    \label{tab:meas-tab}
  \end{table}

Aside from \cref{quest:degree}, the only question I wasn't able to resolve to my own satisfaction was the case of trigonometric polynomials. It is fairly straightforward to show that for any function $g(x_1,\dots,x_n,y,z_1,\dots,z_k)$, 
   \[
     \inf_{y \in \Zb}\sup_{\zbar \in \Zb^k} g(\xbar,y,\zbar) = \inf_{y \in \Rb}\sup_{\beta \in \Rb,\zbar \in \Rb^{k}}\inf_{\gamma \in \Rb} \left[  g(\xbar,y,\zbar) + \beta^2\sin^2(\pi y)   - \gamma^2\sum_{i = 1}^{k} \sin^2(\pi z_i) \right]
   \]
   for any $\xbar \in \Rb^n$. Using a standard argument involving Lagrange's four-square theorem (which allows us to convert quantification over $\Nb$ into quantification over $\Zb$ at the cost of nearly quadrupling the number of variables), this allows us to modify the proof of \cref{thm:polynomial-Vitali} to show that there is a function of the form  $f(x) = \inf_{y \in \Rb}\sup_{ z \in \Rb}\inf_{\nbar \in \Rb^4}\sup_{\beta \in \Rb,\kbar \in \Rb^{280}}\inf_{\gamma \in \Rb} g$ for a trigonometric polynomial $g$ such that whether $f(x)$ is measurable is independent of \ZFC. The fact that something like this should be possible follows from \cref{fact:sine-sin}, but it is unclear if the extra alternation of quantifiers is necessary.

   \begin{quest}\label{quest:trig}
     Does \ZFC\ prove that for every trigonometric polynomial $g$, any function of the form $\inf_\Rb^\ast\sup_\Rb^\ast \inf_\Rb^\ast \sup_\Rb^\ast g$ is Lebesgue measurable?
   \end{quest}

   A positive answer to this would be interesting since it isn't implied by the general statements \ref{lem:semi-sigma}, \ref{cor:three-quant-insufficient-1}, \ref{cor:o-min-meas}, and \ref{prop:local-o-min}.

%


\appendix

\section{An explicit paradoxical decomposition of the sphere}
\label{sec:explicit-BT}

In essentially the same manner as the proof of \cref{prop:Vitali-V-L}, we can build a low-complexity paradoxical decomposition of $S^2$ assuming $\VV = \LL$. The commonly presented proof (such as the proof in \cite{tomkowicz2016banach}, which we will roughly follow) explicitly gives two rotation matrices $\rho$ and $\phi$ which generate a free group $F \subseteq SO(3)$. This is then used to build a paradoxical decomposition of $S^2 \setminus D$, where $D$ is the set of points on the sphere fixed by some element of $F$. Then an uncountability argument is used to find a third rotation $\alpha$ satisfying that the sets $(\alpha^n\cdot D)_{n \in \Nb}$ are pairwise disjoint. While this kind of uncountability argument can be done in a computable way, we can also just be slightly more careful with our choice of rotation matrices, which fits this paper's  general theme of explicitness.

Fix the following two rotation matrices:
\[
  \rho = \frac{1}{5}\begin{pmatrix} 3 & 4 & 0 \\ -4 & 3 & 0 \\ 0 & 0 & 5\end{pmatrix}, \quad\quad
  \phi = \frac{1}{5}\begin{pmatrix} 5 & 0 & 0 \\ 0 & 3 & 4 \\ 0 & -4 & 3 \end{pmatrix}.
\]
These generate a free group $F \subseteq SO(3)$. This can be seen by analyzing the  $\mathrm{mod}\ 5$ behavior of the numerators of entries of products of $\rho$ and $\phi$ or by analyzing the dynamics of the matrices in the $5$-adic numbers, as is done in \cite{Speyer2007}. Let $D$ be the set of points in $S^2$ fixed by some element of $F$. Note that since every product of $\rho$ and $\phi$ (and their inverses) is a rational matrix, we have that every element of $D$ is the normalization of a vector with rational coordinates. Therefore for any $\xbar,\ybar \in D$, there is a quartic (i.e., degree $4$) field extension $K$ of $\Qb$ such that the coordinates of $\xbar$ and $\ybar$ are all elements of $K$.

Now let $\alpha$ be the rotation matrix corresponding to the unit quaternion \(\frac{1+i+j\sqrt[3]{2}+k\sqrt[5]{2}}{\sqrt{2+\sqrt[3]{4}+\sqrt[5]{4}}}\). In other words, $\alpha$ represents a rotation of $2 \arccos\left(\frac{1}{\sqrt{2+\sqrt[3]{4}+\sqrt[5]{4}}}  \right)$ radians about the axis $\langle 1,\sqrt[3]{2},\sqrt[5]{2} \rangle$. By a short field-theoretic argument,\footnote{If $K$ is a quartic field extension of $\Qb$, then $K(\sqrt[3]{2})$ cannot contain $\sqrt[5]{2}$ by the multiplicativity of degrees of field extensions.
} the set $\{1,\sqrt[3]{2},\sqrt[5]{2}\}$ is linearly independent over any quartic field extension $K$ of $\Qb$. In particular, this implies that for any distinct $\xbar,\ybar \in D$, we have that the inner products of $\xbar$ and $\langle 1,\sqrt[3]{2},\sqrt[5]{2} \rangle$ and of $\ybar$ and $\langle 1,\sqrt[3]{2},\sqrt[5]{2} \rangle$ are not equal. Since these inner products are preserved by $\alpha$, we have that $\xbar \neq \alpha^n \cdot \ybar$ for any $n$. Moreover, $2 \arccos\left(\frac{1}{\sqrt{2+\sqrt[3]{4}+\sqrt[5]{4}}}  \right)$ is not a rational multiple of $\pi$ by another short field-theoretic argument.\footnote{If $\arccos\left(\frac{1}{\sqrt{2+\sqrt[3]{4}+\sqrt[5]{4}}}  \right)$ is a rational multiple of $\pi$, then $K = \Qb(\sqrt{2+\sqrt[3]{4}+\sqrt[5]{4}})$ is a subfield of a cyclotomic extension of $\Qb$, but $K$ has $\Qb(\sqrt[3]{4}+\sqrt[5]{4})$ as a subfield and the minimal polynomial of $\sqrt[3]{4}+\sqrt[5]{4}$ is $x^{15} - 20 x^{12} - 12 x^{10} + 160 x^9 - 1440 x^7 - 640 x^6 + 48 x^5 - 8640 x^4 + 1280 x^3 - 1920 x^2 - 3840 x - 1088$, so $\Qb(\sqrt[3]{4}+\sqrt[5]{4})$ has the non-abelian Galois group $S_3 \times \mathrm{Aut}(D_5)$ \cite{lmfdb:15.1.7174453500000000000000.1}.
} Finally, $\langle 1,\sqrt[3]{2},\sqrt[5]{2} \rangle$ is not the fixed axis of any element of $F$, so we have that $\xbar \neq \alpha^n\cdot \xbar$ for any $n > 0$ and $\xbar \in D$. Therefore the sets $(\alpha^n\cdot D)_{n \in \Nb}$ are pairwise disjoint.

We have an effective procedure for listing the elements of $D$. Fix a computable enumeration\footnote{All we need of this enumeration of $F_2$ is that multiplication is computable and that we can compute the reduced word representation of $f(n)$ uniformly in $n$. To accomplish this it would be sufficient to list the elements of $F_2$ in lexicographic order. Alternatively, we could literally just write $n$ in base $4$ and interpret it as a word in the two generators of $F_2$ and their inverses to get $f(n)$ for $n > 0$.} $f : \Nb \to F_2$ of the free group on two generators satisfying that $f(0) = e$. Let $g : F_2 \to F$ be the isomorphism taking the two generators of $F_2$ to $\rho$ and $\phi$.

\begin{lem}\label{lem:D-Sigma-0-2}
  $D$ is a $\lSigma^0_2$ set.
\end{lem}
\begin{proof}
  The function $g \circ f$ is computable (in the sense that we can write a computer program that computes the elements of the entries of each matrix $g(f(n))$). Furthermore, there is a uniform procedure for computing the fixed axis of the non-trivial matrices in $F$, so we can find two computable function $h_1$ and $h_2$ that produce the coordinates of the two fixed points of $g(f(n))$ for $n > 0$. (For $n = 0$, we can just let $h_1(0) = h_1(1)$ and $h_2(0) = h_2(1)$.) %
  Membership of $h_1(n)$ and $h_2(n)$ in any rational open cube $(a_1,b_1)\times (a_2,b_2)\times (a_3,b_3)$ is decidable.\footnote{Although it's possible to show this directly, it's also true that this follows from the Tarski--Seidenberg theorem, since it gives an algorithm for deciding the truth of arbitrary first-order sentences (without parameters) in the structure $(\Rb,+,\cdot)$.} Therefore we can uniformly enumerate the set $\{k : h_1(n) \notin N(\Rb^3,k)\}$ for each $n$ (and likewise for the analogous set with $h_2$). Therefore we get that the set $U \subseteq \Rb^3 \times \Nb = \{\langle \xbar,2n \rangle : \xbar \neq h_1(n)\}\cup\{\langle \xbar,2n+1 \rangle : \xbar \neq h_2(n)\}$ is semicomputable. Finally, $D = \pi((\Rb^3 \times \Nb)\setminus U)$, so $D$ is $\lSigma^0_2$.
\end{proof}

Let $D^\ast = \{\alpha^n : n \geq 0\} \cdot D$.

\begin{cor}\label{cor:F-D}
  The set $D^\ast$ is $\lSigma^0_2$.
\end{cor}
\begin{proof}
  We have that the function $h : \Rb^3 \times \Nb \to \Rb^3$ defined by $h(\xbar,n) = \alpha^{-n}\cdot \xbar$ is computable. By \cref{fact:function-sub}, 
  this implies that the preimage $h^{-1}(D) \subseteq \Rb^3 \times \Nb$ is $\lSigma^0_2$. Finally, $D^\ast =\pi(g^{-1}(D))$ (where $\pi$ is the projection from $\Rb^3 \times \Nb$ to $\Rb^3$), so by \cref{fact:Bool}, we have that $D^\ast$ is $\lSigma^0_2$.
\end{proof}

Since any $\lPi^0_1$ set is also $\lPi^0_2$, we also have the following corollary.

\begin{cor}\label{cor:S-D}
  $S^2 \setminus D$ and $S^2 \setminus D^\ast$ are $\lPi^0_2$ sets.
\end{cor}
\begin{proof}
  It is straightforward to show that $S^2$ is $\lPi^0_1$. Therefore it is $\lPi^0_2$ as well. By \cref{fact:Bool} and \cref{cor:F-D}, we get that $S^2 \setminus D$ and $S^2 \setminus D^\ast$ are $\lPi^0_2$.
\end{proof}

In a similar manner to the proofs of \cref{lem:Vitali-equiv-basic} and \cref{cor:F-D}, we get the following.

\begin{lem}\label{lem:BT-equiv}
  The equivalence relation on $S^2$ defined by $\xbar \in F \cdot \ybar$ is $\lSigma^0_2$.
\end{lem}
\begin{proof}
  Again using the $f$ and $g$ from the proof of \cref{lem:D-Sigma-0-2}, we have that the map $h : \Rb^3 \times \Rb^3 \times \Nb \to \Rb^3 \times \Rb^3$ defined by $h(\xbar,\ybar,n) = \langle \xbar,g(f(n))\cdot \ybar \rangle$ is computable. The set $\Delta = \{\langle \xbar,\ybar \rangle \in \Rb^3 \times \Rb^3 : \xbar = \ybar\}$ is $\lPi^0_1$, so by \cref{fact:function-sub}, 
  the preimage $h^{-1}(\Delta)$ is $\lPi^0_1$. The set we're after is then $(S^2\times S^2) \cap (\Delta \cup \pi(h^{-1}(\Delta)))$ (where $\pi$ is the projection from $\Rb^3\times \Nb$ to $\Rb^3$), which is $\lSigma^0_2$ by \cref{fact:Bool}.
\end{proof}

\begin{prop}\label{prop:BT-V-L}
  There is a semicomputable set $U \subseteq \Rb^3 \times \Rb^2 \times \Nb$ such that, assuming $\VV = \LL$, 
  \(
    \{\xbar : (\forall y \in \Rb)(\exists z \in \Rb)(\forall n \in \Nb)\langle \xbar,y,z,n \rangle \in U\}
  \)
  is a $\lPi^1_2$ transversal of the equivalence relation $x \in F \cdot y$ on $S^2 \setminus D$.
\end{prop}
\begin{proof}
  By \cref{cor:S-D} and \cref{lem:BT-equiv}, we have that the restriction of this equivalence relation to $S^2 \setminus D$ is $\lDelta^0_3$ (and therefore $\lDelta^1_2$). The result then follows from \cref{lem:transversal}.
\end{proof}

Fix some such $\lPi^1_2$ transversal $T$ of the equivalence relation $x \in F \cdot y$ on $S^2 \setminus D$.

\begin{lem}
  There is a paradoxical decomposition $A_1$, $A_2$, $A_3$, $A_4$ of $F_2$ such that the sets $f^{-1}(A_i)$ are each computable (i.e., $\lDelta^0_1$ subsets of $\Nb$).
\end{lem}
\begin{proof}
  Let $\tau$ and $\sigma$ be the two generators of $F_2$. The four sets $(A_i)_{i \leq 4}$ given in \cite[Thm.~4.2]{tomkowicz2016banach} are defined as follows:
  \begin{itemize}
  \item $x \in A_1$ if and only if the first letter in the reduced word of $x$ is $\tau$.
  \item $x \in A_2$ if and only if the first letter in the reduced word of $x$ is $\tau^{-1}$.
  \item $x \in A_3$ if and only if the first letter in the reduced word of $x$ is $\sigma$ or $x$ is of the form $\sigma^{-n}$ for some $n \geq 0$.
  \item $x \in A_4$ if and only if the first letter in the reduced word of $x$ is $\sigma^{-1}$ and $x$ is not of the form $\sigma^{-n}$ for some $n$.
  \end{itemize}
  Each of these is clearly decidable from the word representation of $x$. This implies that each of the sets $f^{-1}(A_i)$ is decidable, since $f$ is a computable enumeration of $F_2$.
\end{proof}

Now by the standard argument, we get that the sets $(g(A_i)\cdot T)_{i \leq 4}$ are a paradoxical decomposition of $S^2 \setminus D$.

\begin{prop}\label{prop:BT-pieces-Pi-1-2}
  The sets $(g(A_i)\cdot T)_{i \leq 4}$ are each $\lPi^1_2$.
\end{prop}
\begin{proof}
  The function $h(\xbar,n) = g(f(n))^{-1}\cdot \xbar$ is computable, so the set $h^{-1}(T) = \{\langle g(f(n))\cdot\xbar,n \rangle : \xbar \in T,~n \in \Nb\} \subseteq \Rb^3 \times \Nb$ is $\lPi^1_2$ by \cref{fact:function-sub}.
  Each of the sets $B_i = \{\langle \xbar, n \rangle : \xbar \in \Rb^3,~n \in f^{-1}(A)\}$ is $\lDelta^0_1$. Therefore, the sets $h^{-1}(T) \cap B_i$ are each $\lPi^1_2$. By \cref{fact:Bool}, this implies that the projections $\pi(h^{-1}(T) \cap B_i) \subseteq \Rb^3$ are $\lPi^1_2$.
\end{proof}

\begin{cor}\label{cor:paradox-Pi}
  There is a specific sequence of $16$ $\lPi^1_2$ sets which, under $\VV = \LL$, are a paradoxical decomposition of $S^2$.
\end{cor}
\begin{proof}
  Each of the pieces in the decomposition can be written as an intersect of a computable rotation of $S^2 \setminus D^\ast$ or $D^\ast$, $g(A_i)\cdot T$ for some $i \leq 4$, and another computable rotation of $S^2 \setminus D^\ast$ or $D^\ast$, giving $16$ combinations. By \cref{fact:Bool} and Corollaries~\ref{cor:F-D} and \ref{cor:S-D}, these pieces are all $\lPi^0_2$.
\end{proof}

This can almost certainly be improved to four pieces, the known optimal size of a paradoxical decomposition of $S^2$, by directly adapting the proof of \cite[Thm.~4.5]{tomkowicz2016banach}. There is also no fundamental extra subtlety in writing out a $\lPi^1_2$ decomposition of the full unit ball, rather than just the unit sphere.

%


\bibliographystyle{plain}
\bibliography{../ref}

\end{document}